\documentclass{amsart}
\usepackage[foot]{amsaddr}
\usepackage{bbm,amsmath,amsthm,amsfonts, amssymb,mathdots,enumitem,mathrsfs,color,url}
 \newtheorem{thm}{Theorem} [section]
 \newtheorem*{theorem*}{Theorem}
\newtheorem{definition}[thm]{Definition}
\newtheorem{lem}[thm]{Lemma}
 \newtheorem{prop}[thm]{Proposition}
\newtheorem{cor}[thm]{Corollary}

 \newtheorem{prob}[thm]{Problem}

\theoremstyle{remark}
\newtheorem{rem}[thm]{Remark}
  \newcommand{\mQ}{\mathcal{Q}}
  \newcommand{\mS}{\mathcal{S}}
\newcommand{\mO}{\mathcal{O}}  
\newcommand{\mB}{\mathcal{B}}

 \newcommand{\mL}{\mathcal{L}}

\newcommand{\mE}{\mathcal{E}}
\newcommand{\mF}{\mathcal{F}}
  \newcommand{\mN}{\mathcal{N}}

 \newcommand{\bP}{\mathbb{P}} 
 \newcommand{\bF}{\mathbb{F}}  
 \newcommand{\fO}{\mathfrak{O}}  
 \newcommand{\ep}{\epsilon} 

\newcommand{\bbbm}{ \begin{bmatrix}} \newcommand{\bebm}{\end{bmatrix}}
\newcommand{\bbm}{ \begin{pmatrix}} \newcommand{\bem}{\end{pmatrix}}
\newcommand{\bbsm}{ \left( \begin{smallmatrix}} \newcommand{\besm}{\end{smallmatrix} \right)}
\newcommand{\beq}{\begin{equation}}      \newcommand{\eeq}{\end{equation}}
\newcommand{\beqn}{\begin{eqnarray}}      \newcommand{\eeqn}{\end{eqnarray}}
\newcommand{\beqs}{\begin{eqnarray*}}      \newcommand{\eeqs}{\end{eqnarray*}}
\newcommand{\bep}{\begin{proof}}      \newcommand{\eep}{\end{proof}}

\begin{document}
\title{Incidence of lines, points and planes in $PG(3,q)$ with respect to the twisted cubic}
\author{Krishna Kaipa$^{1,\ast}$}
\address{$^1$Department of Mathematics, Indian Institute of Science Education and Research, Pune, Maharashtra, 411008 India.}
\author{Puspendu Pradhan$^2$}
\address{$^2$Department of Mathematics, Indian Institute of Technology Bombay, Mumbai 400076, India.}
\address{$^{\ast}$Corresponding author}
\email{$^1$kaipa@iiserpune.ac.in, $^2$puspendupradhan1@gmail.com}
%\keywords{Reed-Solomon codes, covering radius, deep holes}
\date{}
\begin{abstract} 
We consider the orbits of the group  $G=PGL_2(q)$ on the points, lines and planes of the projective space $PG(3,q)$ over a finite field $\bF_q$ of  characteristic  different from $2$ and $3$. The points of $PG(3,q)$ can be identified with projective space of binary cubic forms, and the set  $\mL$ of lines of $PG(3,q)$ can be thought of as pencils of cubic forms. The action of  $G$ on $PG(1,q)$ naturally induces an action of $G$ on binary cubic forms $f(X,Y)$. The points of $PG(3, q)$ decompose into five $G$ orbits.  The $G$ orbits on $\mL$ were recently obtained by the authors. Let $\mathcal I$ be the subset of $\mL \times PG(3,q)$ consisting of pairs $(L,P)$ where $L$ is a line incident with the point $P$.  The decomposition of $\mL \times PG(3,q)$ into $G \times G$ orbits yields a partition of $\mathcal I$. The problem that we solve in this work is  to determine the sizes of the corresponding parts of $\mathcal I$. 
\end{abstract}
 \maketitle
\section{Introduction} \label{S1}
Let $PG(n-1,q)$ denote the projective space $PG(\bF_q^n)$ over a finite field $\bF_q$.  Let $V_m$ denote the $(m+1)$-dimensional vector space over $\bF_q$, consisting  of degree $m$ homogeneous polynomials $f(X,Y)$ with coefficients in $\bF_q$. We also refer to elements of $V_3$ and $V_4$ as  (binary) cubic forms, and quartic forms respectively.  The group  
$GL_2(q)$ acts on $V_m$ by
\[ (g \cdot f)(X,Y) = \det(g)^{-m} f(dX-bY, aY-cX), \qquad g = \bbsm a&b \\ c&d \besm. \]
This action induces an action of the  group $G=PGL_2(q)$ on $PG(V_m)$.
In this work we assume char$(\bF_q) \neq 2, 3$ and $q>4$. We identify the points of $PG(3,q)$ with the projective space $PG(V_3)$ of binary cubic forms 
\[  f(X,Y)=z_0 Y^3 - 3 z_1 Y^2X +3 z_2 YX^2 -z_3 X^3. \]
The twisted cubic $C$ in $PG(3,q)$ is the image of the map  $(s,t) \hookrightarrow (Xt-Ys)^3$ from $PG(1,q)$ to $PG(3,q)$. The subgroup of $PGL_4(q)$ that preserves $C$ is the isomorphic image of $G$ in $PGL_4(q)$ given by the aforementioned action of $G$ on $PG(3,q)=PG(V_3)$.
It is well known \cite{Hirschfeld3} that there are five  $G$-orbits $O_1, \dots, O_5$ of points of $PG(3,q)$ (see Lemma \ref{cubic}).
 We denote the set of lines of $PG(3,q)$ by $\mL$. A line $L$ of $PG(3,q)$ is a pencil of cubic forms. The pencil $L$ is called non-generic if it either intersects $C$, or if all forms in $L$ have a common linear factor. The non-generic lines can be naturally partitioned into eight parts 
(\cite[p.236]{Hirschfeld3}). The decomposition of these parts into ten  $G$-orbits  was obtained by    Davydov, Marcugini and Pambianco \cite{DMP1}, and  G\"unay and Lavrauw in \cite{GL}, and also by Blokhuis, Pellikaan and S\H{o}nyi in \cite{BPS}. In the work \cite{Ceria_Pavese},  Ceria and  Pavese obtain  the decomposition of the class of  generic lines into $G$-orbits when char$(\bF_q)=2$.  When char$(\bF_q)\neq 2$, the decomposition of the class of generic lines into $G$-orbits was obtained in  our recent work \cite{KPP,KP} based on the decomposition of $PG(V_4)$ into $G$-orbits. There are  $(2q-3+ \mu)$-orbits of generic lines where $q \equiv \mu \mod 3$ and $\mu \in  \{\pm 1\}$. \\

 The $G\times G$ orbits of  $\mL \times PG(3,q)$ are of the form $\fO_\alpha \times O_i $ where $O_i$ is a $G$-orbit of $PG(3,q)$ and $\fO_\alpha$ is a $G$-orbit on $\mL$. Let $\mathcal I$ denote the subset of $\mL \times PG(3,q)$ consisting of incident pairs:
\[ \mathcal I = \{(L,P) \in \mL \times PG(3,q) : P \in L\}. \]
% We may think of $\mathcal I$ as the $\bF_q$-points of the projective tautological bundle over the Grassmannian of lines of the projective $3$-space $\bP^3$ over $\bF_q$. 
If  $PG(3,q) = \cup_{i=1}^5 O_i$ and $\mL=\cup_{\alpha} \fO_\alpha$ is the decomposition of these sets into $G$-orbits, then $\cup_{\alpha, i} \, \fO_\alpha \times O_i$ is the decomposition of  $\mL \times PG(3,q)$ into $G \times G$ orbits.  A natural problem is to determine the sizes of the parts of the partition of $\mathcal I$ determined by:
\[ \mathcal I = \cup_{\alpha,i} \,  \mathcal I \cap (\fO_\alpha \times O_i).  \]
If $L$ is an element of $\fO_\alpha$ and $\mathcal S$ the set of points of $L$, then  $|(\fO_\alpha \times O_i) \cap \mathcal I|$  clearly equals $|\fO_\alpha| |\mathcal S \cap O_i|$.  Therefore, this  problem is equivalent to the following problem:
\begin{prob} \label{prob1} For each orbit $\fO$ of $\mL$, decompose the set $\mathcal S$ of $(q+1)$ points of a fixed line $L$ of $\fO$ as $\mathcal S = \cup_{i=1}^5 \mathcal S \cap O_i$.
\end{prob}
The planes of $PG(3,q)$ also decompose into five $G$-orbits $\mN_1, \dots, \mN_5$ (\cite[p.234]{Hirschfeld3}).
There is a dual  problem closely related to the  Problem \ref{prob1}:
\begin{prob} \label{prob2} For each orbit $\fO$ of lines of $PG(3,q)$, decompose the set $\mathcal{P}$ of $(q+1)$ planes containing a fixed line $L$ of $\fO$ as $\mathcal P = \cup_{i=1}^5 \mathcal P \cap \mathcal \mN_i$.
\end{prob}
When the characteristic of the field $\bF_q$ is not equal to $3$, there is an $G$-invariant bijective correspondence $P \leftrightarrow P^{\perp}$  between points and planes of $PG(3,q)$ and an $G$-invariant involution $L \leftrightarrow L^{\perp}$  on the  lines of $PG(3,q)$. Here $\perp$ denotes the polar dual with respect to the polarity $\Omega_3$ on $\bP(V_3)$ (see \S \ref{S2}). Thus, the $G$-orbits of planes of $PG(3,q)$ are given by $\mN_i=O_i^\perp$ for $i=1, \dots, 5$.  For char$(\bF_q) \neq 3$, 
Problem \ref{prob2} is equivalent to Problem \ref{prob1}, because for $L$ and $\mathcal P$ as in Problem \ref{prob2}, we have $\mathcal P \cap O_i^\perp = (\mathcal S \cap O_i)^\perp$ where $\mathcal S=\mathcal P^\perp$ is the set of $(q+1)$ points of the line $L^\perp$. Consequently, we do not treat Problem \ref{prob2} as being different from Problem \ref{prob1}, and focus only on Problem \ref{prob1}.

Problem \ref{prob1} has  been solved for the ten orbits of non-generic lines by Davydov, Marcugini and Pambianco in \cite{DMP3,DMP5}, G\"unay and Lavrauw in \cite{GL}.
%, for example  in \cite[Table 2]{DMP3} and \cite[Table 1]{GL}.  
For all lines of $PG(3,q)$, Problem \ref{prob1} was solved in characteristic $2$ by Ceria and Pavese in \cite{Ceria_Pavese}. When $\text{char}(\mathbb{F}_q) = 3$, Problems \ref{prob1} and \ref{prob2} have been independently solved for all orbits of non-generic lines by A. Davydov, S. Marcugini, and F. Pambianco in \cite{DMP3, DMP5}. The Problems \ref{prob1} and \ref{prob2} for all orbits of generic lines in characteristic $3$ have been solved in our earlier work \cite{KP}. In the work \cite{DMP4}, Problem \ref{prob1} has been solved for  some of the $G$-orbits of generic lines.  In this work, we solve the problem  for each of the $(2q-3+\mu)$ orbits of generic lines when char$(\bF_q) \neq 2,3$.

\subsection{Statement of our main results} 
We begin with  some notation required to state our result. We recall that $V_4$ is the vector space of binary quartic forms $\varphi  = z_0 Y^4-4z_1Y^3X+6z_2Y^2X^2-4z_3 YX^3 +z_4 X^4$ over $\bF_q$. The action of  $GL_2(q)$  on such forms is as  $g \cdot \varphi = \varphi(dX-bY,aY-cX)/\det(g)^4$ where $g  = \bbsm a&b \\ c & d \besm$. There are two fundamental $GL_2(q)$-invariants of a quartic form:
\[I(\varphi)=(z_0z_4-4z_1z_3 + 3z_2^2)/3,\text{ and}  \quad J(\varphi)=  \det \bbsm z_0 & z_1 & z_2 \\ z_1 & z_2 & z_3 \\z_2 & z_3 & z_4 \besm.\]
The discriminant $\Delta(\varphi) = I^3(\varphi) - J^2(\varphi)$ equals $0$ if and only if $\varphi$ has a linear factor of multiplicity greater than one over some extension field of $\bF_q$. 
 Let $\mF^+_\Delta$ denote 
the subset of the projective space of quartic forms $\varphi$ for which $\Delta(\varphi) \neq 0$ and $I(\varphi)$ is a square in $\bF_q$.
There is a (generically) 2-to-1  $G$-equivariant correspondence between the set of generic lines of $\bP(V_3)$ on one hand, 
and $\mF^+_\Delta$  on the other hand.
The generic lines of $\bP(V_3)$ are parametrized by points $(z_0, \dots, z_5)$ of $PG(5,q)$ such that $z_5^2 = I(\varphi)$ where $\varphi = z_0 Y^4-4z_1Y^3X+6z_2Y^2X^2-4z_3 YX^3 +z_4 X^4$. The two values  $\pm \sqrt I_\varphi$ of $z_5$ correspond to a pair of polar dual lines $L, L^\perp$. We refer the reader to \S \ref{S2} and \S \ref{S3} for  details on binary quartic forms and their classification into $G$-orbits, and the above-mentioned 2-to-1 correspondence between lines of $PG(3,q)$ and quartic forms.\\

The set  of binary quartic forms with non-zero discriminant can be partitioned as $\mF_1 \cup \mF_2 \cup \mF_2' \cup \mF_4 \cup \mF_4'$ based on the factorization of $f$ over $\bF_q[X,Y]$. The irreducible factors (over $\bF_q$) in each case are: 
  \begin{description}
      \item[$\mF_4:$]  $4$ linear forms,
      \item[$\mF_2:$] $2$ linear forms and a quadratic form,
\item[$\mF_1:$] $1$ linear forms and a cubic form,
 \item[$\mF_4':$]  $2$ quadratic forms,
 \item[$\mF_2':$] $1$ quartic form (that is, $f$ is irreducible over $\bF_q$).    \\
  \end{description}

Let 
\[ \eta_L:=\begin{cases} i &\text{if $\varphi_L \in \mF_i$ for $i=1,2,4$},\\
 0 &\text{if $\varphi_L$ is in $\mF_2'$ or $\mF_4'$}. \\\end{cases}\]

Finally, to the generic line $L$ represented by $(z_0, \dots, z_5) \in PG(5,q)$ with $\varphi =\varphi_L= z_0 Y^4-4z_1Y^3X+6z_2Y^2X^2-4z_3 YX^3 +z_4 X^4$  and $z_5^2 = I(\varphi)$,  we associate the elliptic curve $E_L$ defined by 
\[ E_L: \quad  T^2=4 S^3  -g_2(L) S -g_3(L),\]
where, for  $\varphi=\varphi_L$ and $\sqrt I_\varphi :=z_5$,  we have: \begin{align*} \nonumber g_2(L)&=3 \sqrt I_\varphi J_{\varphi}  + \tfrac{15}{4}I_{\varphi}^2,\\
g_3(L)&=\tfrac{-1}{8}( 11  I_{\varphi}^3 + 2 J_{\varphi}^2 + 14 \sqrt{I_\varphi}^3 J_{\varphi} ).\\
\end{align*}

Let $\# E_L(\bF_q)$ denote the number of points of $E_L$ over $\bF_q$ (including the point at infinity). As shown in Remark \ref{rem2}, the quantity $\#E_L(\bF_q)$ only depends on the $G$-orbit of $L$.
\begin{thm} \label{Main_Result}
Let $L$ be a generic line and let $S$ denote the set of points of $L$. Let  $\varphi_L$ be the quartic form associated to $L$, and let $E_L$ denote the Elliptic curve associated to $L$ as above.  Then
\begin{enumerate}
\item $|\mS \cap O_1|=0$,
\item $|\mS \cap O_2|= \eta_L$,
\item $ |\mS \cap O_3|=\tfrac{\# E_L(\bF_q) -3\eta_L}{3}$,
\item $|\mS \cap O_4| = q+1-\tfrac{\# E_L(\bF_q)+\eta_L}{2}$,
\item $|\mS \cap O_5|= \tfrac{\# E_L(\bF_q)}{3}$. 
\end{enumerate}
\end{thm}
The quantity $\tfrac{\#E_L(\bF_q)}{3}$ which appears in the theorem above is an integer, as shown in Remark \ref{rem1}.  The rest of this paper is organized as follows. In Section \S \ref{S2}, we describe the geometric setup of the problem. In Section \S \ref{S3} we recall the $G$-orbits of generic lines of $PG(3,q)$. In Sections \S \ref{S4} and \S \ref{S5}, we develop the tools need to obtain the main result, which we prove in Section \S \ref{S6}.

\section{Geometric Setup} \label{S2}
In this section $F$ denotes an arbitrary field  of characteristic different from $2, 3$. Let $e_1, e_2$ denote the standard basis of the vector space $F^2$, and let $V$ denote the dual vector space, with dual basis  $X,Y$. The space $\wedge^2V$ is one dimensional with basis $X \wedge Y$, and for a positive integer $m$, the one-dimensional space which is the tensor product of $\wedge^2 V$ with itself $m$ times will be denoted $(\wedge^2V)^{\otimes m}$.
The dual space of this one-dimensional vector space is $(\wedge^2 V^*)^{\otimes m}$ spanned by $(e_1\wedge e_2)^{\otimes m}$. Let $V_m=\text{Sym}^m(V)$ denote the vector space  of degree $m$ binary forms over $F$.  The group  $GL_2(F)$ acts on $V_m$ by
\beq \label{eq:Vm} (g \cdot f)(X,Y) =  \det(g)^{-m}\, f(dX-bY, aY-cX), \qquad g = \bbsm a&b \\ c&d \besm. \eeq

The $(m+1)$-tuple of forms 
\[ \mathcal B_m: \quad (Y^m, -\tbinom{m}{1} Y^{m-1}X, \tbinom{m}{2} Y^{m-2}X^2,  \dots, (-1)^mX^m), \]
forms a basis of $V_m$ if and only if all the binomial coefficients $\tbinom{m}{i}$ are nonzero in $F$. 
Since char$(F) \neq 2,3$, this condition is satisfied for $m \in \{1,\dots,4\}$. We will use this basis for the vector spaces $V_1, \dots, V_4$.  For $m \in \{1,\dots,4\}$,  the degree $m$-rational normal curve $C_m$ in $\bP(V_m)$ is given by the embedding $\nu_m:\bP(V_1) \hookrightarrow \bP(V_m)$ defined by  $\nu_m(Xt-Ys)=(Xt-Ys)^m$. In terms of the bases $\mB_m$, the coordinate description of $\nu_m$ is $(s,t) \mapsto (s^m, s^{m-1}t, \dots, t^m)$. The \emph{osculating hyperplane} $O_{(s,t)}$ to $C_m$ at a point $(Xt-Ys)^m$ consists of all elements of $V_m$ which are divisible by $(Xt-Ys)$.\\

It will be useful to record the matrices $g_m$ representing the action of $g=\bbsm a & b\\ c & d \besm \in GL_2(F)$ on $V_m$ for $m \in \{1, \dots, 4\}$, with respect to the basis $\mB_m$:
\begin{align} \label{eq:g_4} 
 g_1&= \det(g)^{-1} \, g,\quad g_{2}=\det(g)^{-2}\,  \bbm a^2 & 2 ab & b^2\\ac&ad+bc&bd\\c^2&2cd&d^2 \bem, \\
\nonumber g_{3}&=\det(g)^{-3}\, \bbm a^3 &3a^2b&3ab^2&b^3\\
a^2c &a(ad+2bc) &b(bc+2ad)&b^2d\\
ac^2 &c(bc+2ad)&d(ad+2bc)&bd^2\\
c^3 &3c^2d&3cd^2&d^3\bem,   \\
 \nonumber g_{4}&=\det(g)^{-4}\,\bbm a^4 &4a^3b&6a^2b^2&4ab^3&b^4\\
a^3c &a^2(ad+3bc)  & 3ab(ad+bc)&b^2(bc+3ad)&b^3d\\
a^2c^2 &2ac(bc+ad)& (ad+bc)^2+2abcd&2bd(ad+bc)&d^2b^2\\
c^3a &c^2(bc+3ad) & 3cd (ad+bc) &d^2(ad+3bc) &d^3b\\
c^4 &4c^3d&6c^2d^2&4cd^3&d^4 \bem. \end{align}

For $m \in \{1, \dots, 4\}$, let  $\Omega_m$  be the non-degenerate  bilinear form on $V_m$ whose matrix with respect to the basis $\mathcal B_m$ is 
\beq  A_1 = \bbsm &1\\-1&\besm , \quad A_2 = \bbsm &&1\\&-2&\\1&& \besm, \quad A_3 = \bbsm &&&1\\&&-3&\\&3&&\\-1&&& \besm, \quad A_4= \bbsm &&&&1\\&&&-4&\\&&6&&\\&-4&&&\\1&&&& \besm. \eeq
The bilinear forms  $\Omega_2$ and $\Omega_4$ are symmetric of parabolic  type, whereas $\Omega_1, \Omega_3$ are symplectic. It is readily checked that 
\beq  \label{eq:Aj_invariance} 
g_m^{-\top} A_m g_m^{-1} = \det(g)^m A_m,  \text{ equivalently } 
g \cdot \Omega_m = \det(g)^{m} \Omega_m.\eeq
The non-degenerate bilinear forms $\Omega_m$ (for $1 \leq m \leq 4$) on $V_m$ give a `polarity' on $\bP(V_m)$. The polar dual of a $r$-dimensional linear subspace $S$  of $\bP(V_m)$ represented  by the $(m-r-1)$-dimensional subspace given by  $\{g \in V_m \colon \Omega_m(f,g)=0 \text{ for all $f \in S$}\}$. Two examples of this which are important in this work are
\begin{enumerate}
    \item  The hyperplane which is the polar dual of $(Xt-Ys)^m$ under $\Omega_m$ is the osculating hyperplane  $O_{(s,t)}$ to $C_m$ at $(Xt-Ys)^m$.
    
    \item For $m=3$, the polar dual of each line $L$ of $\bP(V_3)$ is another line denoted $L^{\perp}$. For example, a line $L$ meets $C$ (say at $(Xt-Ys)^3$) if and only if $L^{\perp}$ lies in the osculating plane $\mO_{(s,t)}$.
\end{enumerate}

\subsection{ $(\text{Sym}^2 V_2)^*$ and $\wedge^2 V_3$ and  as $GL_2(F)$-modules} \hfill \\
The action of $GL_2(F)$ on $V_2$ and  $V_3$ naturally induces actions of $GL_2(F)$ on the space $(\text{Sym}^2 V_2)^*$ of symmetric bilinear forms on $V_2$, and on $\wedge^2 V_3$ which is the second exterior power of $V_3$. We are interested in the action of $GL_2(F)$ on $\wedge^2 V_3$, because of the Klein representation of the lines of $\bP(V_3)$ as points of the Klein quadric $\mQ$ in $\bP(\wedge^2 V_3)$. We will show that the projective spaces $\bP(\wedge^2 V_3)$ and $\bP( (\text{Sym}^2 V_2)^*)$ are isomorphic as $PGL_2(F)$-modules. Under this isomorphism, a line of $\bP(V_3)$ is generic if and only if the corresponding symmetric bilinear form on $V_2$ is non-degenerate. This isomorphism will be used in our solution of Problem \ref{prob1}.
\subsubsection*{Decomposition of $(\text{Sym}^2 V_2)^*$ as a $GL_2(F)$-module} \hfill \\ Let $(\text{Sym}^2 V_2)^*$ denote the $6$-dimensional vector space  of symmetric bilinear forms on $V_2$. Given $\varphi \in V_4$, consider the symmetric bilinear form on $V_2$ given by 
\[ \langle f_1, f_2 \rangle_\varphi=\Omega_4(f_1f_2, \varphi),\] 
for $f_1, f_2 \in V_2$. Since $\Omega_4$ is non-degenerate, we get an injective linear map $V_4 \hookrightarrow (\text{Sym}^2 V_2)^*$. For $\varphi\in V_4$ with coordinates $(z_0, \dots, z_4)$ with respect to the basis $\mB_4$, the matrix of the bilinear form $\langle \cdot, \cdot \rangle_\varphi$ with respect to the basis $\mB_2$ is  
\beq \label{eq:Mphi}M_\varphi= \bbm z_4 & -2 z_3 & z_2 \\ -2 z_3 & 4 z_2 & -2 z_1 \\ z_2 & -2 z_1 & z_0 \bem. \eeq 
The $5$-dimensional subspace $\{\langle \,, \rangle_\varphi \colon \varphi \in V_4\}$  of $(\text{Sym}^2 V_2)^*$  is a complement of the one dimensional subspace spanned by $\Omega_2$ because $M_\varphi \neq A_2$ for any $\varphi \in V_4$.
For $ g \in GL_2(F)$, we have $\langle g^{-1}\cdot f_1,g^{-1}\cdot f_2\rangle_\varphi =\Omega_4((g^{-1}\cdot f_1)(g^{-1}\cdot f_2),\varphi)$ equals
\[\Omega_4( g^{-1} \cdot (f_1f_2),\varphi)=\det(g)^4  \, \Omega_4(f_1f_2, g\cdot\varphi) = \det(g)^4 \, \langle f_1,f_2 \rangle_{g \cdot \varphi}. \]
Thus,  $g \cdot \langle \cdot, \cdot \rangle_\varphi= \det(g)^4 \langle \cdot, \cdot \rangle_{g \cdot \varphi}$, or in coordinates:
 \beq \label{eq:M_invariance}
 M_{g \cdot \varphi} = \det(g)^{-4} g_{2}^{-\top} M_\varphi g_{2}^{-1}.\eeq
Thus, we get a $GL_2(F)$-equivariant map
\[ (\wedge^2V^*)^{\otimes 4} \otimes V_4 \hookrightarrow  (\text{Sym}^2 V_2)^*, \quad (e_1\wedge e_2)^{\otimes 4} \otimes \varphi \mapsto \langle \cdot, \cdot \rangle_\varphi. \]
Since $g \cdot \Omega_2= \det(g)^2 \,  \Omega_2$, we conclude that $(\text{Sym}^2 V_2)^*$ is isomorphic as a $GL_2(F)$-module to
\beq  \label{eq:sym2} (\text{Sym}^2 V_2)^* \simeq (\wedge^2 V^*)^{\otimes 2} \,\oplus\,  (\wedge^2V^*)^{\otimes 4} \otimes V_4  = (\wedge^2V^*)^{\otimes 4} \otimes \left( (\wedge^2 V)^{\otimes 2} \oplus   V_4 \right). \eeq
Under this isomorphism $(e_1 \wedge e_2)^4\otimes( \lambda (X\wedge Y)^2 + \varphi) \mapsto \langle \, , \rangle_\varphi + \lambda\,  \Omega_2$.
\subsubsection*{Decomposition of $\wedge^2 V_3$ as a $GL_2(F)$-module} \hfill \\
The basis $\mB_3: =(b_0, \dots, b_3)$ gives a basis 
\[ \wedge^2 \mathcal B_3 : (b_{ij}= b_i \wedge b_j : \quad 0 \leq i < j \leq 3), \]
of $\wedge^2V_3$.
Let  $(p_{01}, p_{02}, p_{03},p_{12}, p_{13}, p_{23})$ denote the  Pl\"ucker coordinates on $\wedge^2V_3$ with respect to this basis $\wedge^2 \mathcal B_3$.
The action of $g \in GL_2(F)$ on $V_3$ given by the matrices $g_3$ with respect to the basis $\mB_3$, induces an action on $\wedge^2 V_3$. It will be more convenient to describe the matrix action of $GL_2(F)$ on $\wedge^2 V_3$ with respect to a  different basis: 
\[\mE_5: \; E_0 = b_{01}, E_1=2 b_{02}, E_2= 3 b_{03} + b_{12}, E_3=2 b_{13}
, E_4= b_{23}, E_5=3 b_{03} - b_{12}.\]
The coordinates  $(z_0, \dots, z_5)$ with respect to the basis $\mathcal E_5$ are related to the Pl\"ucker coordinates $p_{ij}$ by
\beq \label{eq:newcoords}  (p_{01}, p_{02}, p_{03}, p_{12}, p_{13}, p_{23})=(z_0, 2z_1,3(z_2+z_5), z_2-z_5, 2z_3,z_4).  \eeq
A direct calculation shows that the action of $GL_2(F)$ on $\wedge^2 V_3$ is represented in the coordinates $(z_0, \dots, z_5)$ by the matrix 
\beq \label{eq:wedge2g4} 
\tilde g_5=\det(g)^{-1} \bbsm g_4 & 0\\0 & \det(g)^{-2} \besm,  \eeq
where $g_4$ (see \eqref{eq:g_4}) represents the action of $GL_2(F)$ on $V_4$ with respect to the basis $\mB_4$. This shows that we have a $GL_2(F)$-equivariant isomorphism 
\beq \label{eq:Phi_isom}\Phi: \wedge^2 V_3 \to  \left( V_4 \,  \oplus \,  (\wedge^2V)^{\otimes 2} \right) \otimes \, \wedge^2V,\eeq 
given in coordinates by 
\[ \Phi(\sum_{i=0}^5 z_i E_i)=  (z_0 Y^4-4z_1Y^3X+6z_2Y^2X^2-4z_3 YX^3 +z_4 X^4)\otimes (X\wedge Y) + z_5  (X\wedge Y)^{\otimes 3}  .\]
Combining this isomorphism $\Phi$ with the isomorphism \eqref{eq:sym2}, we get
a $GL_2(F)$-equivariant isomorphism
\beq \label{eq:Psi_isom} \Psi: \wedge^2 V_3  \simeq   (\text{Sym}^2 V_2)^* \otimes (\wedge^2V)^{\otimes 5}. \eeq
For  $w \in \wedge^2 V_3$ with coordinates $z=(z_0, \dots, z_5)$ with respect to the basis $\mE_5$, the matrix of the symmetric bilinear form $\Psi(w)$ with respect to the basis $\mB_2$ of $V_2$ is $(z_5 A_2 + M_{\varphi})$ where $\varphi=z_0 Y^4 - 4z_1 Y^3X +6 Y^2X^2 -4 z_3 YX^3 +z_4 X^4$. We denote this matrix as 
\beq \label{eq:Mz} M_z = \bbm z_4 & -2 z_3 & z_2 +z_5 \\ -2 z_3 & 4 z_2 -2 z_5 & -2 z_1 \\ z_2+z_5 & -2 z_1 & z_0 \bem. \eeq
The $GL_2(F)$-equivariance of the map $\Psi: \wedge^2V_3 \to (\wedge^2 V)^{\otimes 5} \otimes (\text{Sym}^2 V_2)^*$ shows that 
\beq \label{eq:Mz_equivarince} M_{ \tilde g_5 z} = \det(g)^{-5} g_2^{-\top} M_z  g_2^{-1}.
\eeq
\subsubsection*{$GL_2(F)$-invariants of quartic forms.} \hfill \\
Given $\varphi \in V_4$ 
\[\varphi(X,Y)= z_0 Y^4 -4z_1 Y^3X +6 z_2 Y^2X^2 -4 z_3 Y X^3 +z_4 X^4,\]
we define 
\beq \label{eq:I} I(\varphi) = \Omega_4(\varphi, \varphi)/6 = (z_0z_4-4z_1z_3 +3 z_2^2)/3. \eeq 
Since $g \cdot \Omega_4 = \det(g)^4 \Omega$, we get 
\[ 6 I( g \cdot \varphi) = \Omega_4( g \cdot \varphi, g \cdot \varphi) = (g^{-1} \Omega_4)(\varphi, \varphi) = \det(g)^{-4} \Omega_4(\varphi, \varphi)=6 \det(g)^{-4} I(\varphi).\]

%and it satisfies $I(g \cdot \varphi) = \det(g)^{-4} I(\varphi)$. 
We define  $J(\varphi)$ to be  the cubic form in the coefficients of $\varphi$ given by 
 \beq \label{eq:J} J(\varphi)= \tfrac{1}{4} \det M_{\varphi}. \eeq
Since $M_{g \cdot \varphi} =\det(g)^{-4} g_{2}^{-\top} M_\varphi g_{2}^{-1}$ by \eqref{eq:M_invariance}, and $\det(g_2) = \det(g)^{-3}$, we get:
\[ 4 J(g \cdot \varphi) = \det(M_{g \cdot \varphi}) =\det(g)^{-6} \det(M_\varphi) = 4 J(\varphi).  \]
 We summarize this as:
 \beq \label{eq:IJ_invariance} I(g \cdot \varphi) = \det(g)^{-4} I(\varphi), \quad J(g \cdot \varphi) = \det(g)^{-6} J(\varphi).\eeq
 The quantity $I(\varphi)$ is known as the \emph{apolar invariant} of $\varphi$, and the quantity $J(\varphi)$ is  known as the \emph{Catalecticant} invariant of $\varphi$. The quantity $\Delta(\varphi)=I^3(\varphi)-J^2(\varphi)$ is the \emph{discriminant} of  form $\varphi$. The form $\varphi$ has repeated factors if and only if $\Delta(\varphi)=0$. For a form $\varphi$ with $\Delta(\varphi) \neq 0$, we define 
\beq \label{eq:jvarphi} 1 - \tfrac{1728}{\jmath(\varphi)}= \tfrac{J^2(\varphi)}{I^3(\varphi)}. \eeq
 Since $I^3(g \cdot \varphi)=\det(g)^{-12} I(\varphi)$ and $J^2(g \cdot \varphi)=\det(g)^{-12} J(\varphi)$, it follows that the quantity $\jmath(\varphi)$, only depends on the orbit $GL_2(F)$-orbit of  $\varphi$, and is called the $\jmath$-invariant of $\varphi$. 
It can be shown that two forms $\varphi, \psi$ with nonzero discriminant, have the same $\jmath$-invariant if and only if there is a $g \in PGL_2(\overline F)$ with $g \cdot \varphi = \psi$, where $\overline F$ is an algebraic closure of $F$. As for $PGL_2(F)$-orbits, there can be more than one orbit corresponding to a given value of the $\jmath$-invariant. In the case when $F = \bF_q$, the $PGL_2(q)$-orbits on $\bP(V_4)$ are described in $\S \ref{S3}$.

\subsubsection*{Klein representation of lines of $\bP(V_3)$ } \hfill \\
We recall that for a point $w \in \wedge^2 V_3$ with coordinates $(z_0, \dots, z_5)$ with respect to the basis $\mE_5$, the bilinear form $\Psi(w)$ on $V_2$ is represented by the matrix $M_z=z_5 A_2 +M_\varphi$, where 
$\varphi= z_0 Y^4 -4z_1 Y^3X +6 z_2 Y^2X^2 -4 z_3 Y X^3 +z_4 X^4$. The matrices $M_z$  and $M_\varphi$ are as in \eqref{eq:Mz} and \eqref{eq:Mphi}.
The action of $g \in GL_2(F)$ on $\wedge^2 V_3$ with respect to the basis $\mE_5$ is given by the matrix $\tilde g_5$ of \eqref{eq:wedge2g4}.
We recall that the lines of $\bP(V_3)$ are parametrized by the points of the Klein quadric $\mQ$ in $\bP(\wedge^2 V_3)$. A  line $L$ of $\bP(V_3)$ generated by independent  cubic forms $u(X,Y), v(X,Y) \in V_3$:
Given $u, v \in V_3$ 
\begin{align} \label{eq:uv}
 u(X,Y)&=u_0Y^3-3u_1Y^2X+3u_2YX^2-u_3X^3, \\ 
\nonumber v(X,Y)&=v_0Y^3-3v_1Y^2X+3v_2YX^2-v_3X^3, \end{align}
has Pl\"ucker coordinates 
\[p_{ij} = u_iv_j - u_j v_i, \quad   0 \leq i < j \leq 3. \]
We recall that a point of $\bP(\wedge^2 V_3)$ with Pl\"ucker coordinates $(p_{01}, \dots, p_{23})$ represents a line of $\bP(V_3)$ if and only if it lies on the Klein quadric:
\[\mQ:  p_{01}p_{23} - p_{02}p_{13} +p_{03}p_{12}=0. \]
In terms of the coordinates $(z_0, \dots, z_5)$ with respect to the basis $\mE_5$ of $\wedge^2 V_3$ ( see  \eqref{eq:newcoords}), the Klein quadric $\mQ$ is given by
\beq \label{eq:Q} \mQ: 
(z_0z_4-4 z_1z_3 +3 z_2^2)/3 = z_5^2. \eeq
In terms of the $PGL_2(F)$-equivariant  isomorphism $\Phi$ of \eqref{eq:Phi_isom} (at the projective level), we see that 
\[ \Phi(\mQ) = \{ \varphi + z_5 (X\wedge Y)^{\otimes 2}  \colon z_5^2 = I(\varphi) \}. \]
\begin{lem}
For a line $L$ represented by a pair $(\varphi, z_5)$, we claim
\begin{enumerate}
    \item[$(1)$] $L$ is non-generic if and only the discriminant $\Delta(\varphi)=0$.
    \item[$(2)$] $L$ lies on an osculating plane of $C$ if and only if $J(\varphi)=z_5^3$.
    \item[$(3)$] $L$ intersects $C$  if and only if $J(\varphi)=-z_5^3$.
\end{enumerate}
    \end{lem}
\bep Let $L$ is a line represented by $(\varphi_L, z_5)$. We note that $L$ is contained in the osculating plane $O_{(0,1)}$ to $C$ at the point $X^3$, if and only if its Pl\"ucker coordinates satisfy $p_{01}=p_{02}=p_{03}=0$. Dually, $L$ intersects $C$ at the point $X^3$ if and only if its Pl\"ucker coordinates satisfy $p_{01}=p_{02}=p_{12}=0$. In terms of the representation $(\varphi_L, z_5)$ of $L$, we see that (i) $L$ is contained in the osculating plane $O_{(0,1)}$ if and only if $(\varphi_L, z_5) = (X^2(6z_2 Y^2 - 4 z_3 YX +z_4 X^2),-z_2)$, and (ii)  $L$ intersects $C$ at the point $X^3$ if and only if $(\varphi_L, z_5) = (X^2(6z_2 Y^2 - 4 z_3 YX +z_4 X^2), z_2)$. We note that the quartic form $\varphi_L=X^2(6z_2 Y^2 - 4 z_3 YX +z_4 X^2)$ has $J(\varphi_L) = -z_2^3$ and $I(\varphi_L)=z_2^2$. So the conditions $z_5 =-z_2, z_5=z_2$ are equivalent to $J(\varphi_L)=z_5^3, J(\varphi_L)=-z_5^3$, respectively.\\

We recall that for $g \in GL_2(K)$ where $K \supset F$ is an extension field of $F$, we have 
\[z_5(g \cdot L)=\det(g)^{-3} z_5(L), \quad \varphi_{g \cdot L} = \det(g)^{-1} g \cdot \varphi_L.\]
We also note that 
\[ J(\varphi_{g \cdot L})=J(\det(g)^{-1} g \cdot \varphi_L)=\det(g)^{-3} J(g \cdot \varphi_L)=\det(g)^{-9} J(\varphi_L).\]
Therefore, if $(Xt-Ys)^3$ is a point of $C$ over an extension field $K \supset F$, and if   $g \in GL_2(K)$ satisfies $g \cdot (Xt-Ys)=X$ (equivalently, $g\cdot (s,t)=(0,1)$), then we see that (i) $L$ is contained in the osculating plane $O_{(s,t)}$ if and only if 
\[ J(\varphi_{g\cdot L})= z_5(g \cdot L)^3,\; \text{ equivalently }  \det(g)^{-9} J(\varphi_L) = \det(g)^{-9} z_5(L)^3, \]
and (ii)  $L$ intersects $C$ at the point $(Xt-Ys)^3$ if and only if 
 \[ J(\varphi_{g \cdot L})= -z_5(g \cdot L)^3,\; \text{ equivalently }  \det(g)^{-9} J(\varphi_L) = -\det(g)^{-9} z_5(L)^3, \]
Cancelling the multiplicative factor $\det(g)^{-9}$ from both sides establishes the assertions (2) and (3). As for the assertion (1), we note that $L$ is non-generic if and only if $L$ either meets $C$ or is contained in some osculating plane of $C$. Thus $L$ is non-generic if and only if $J(\varphi)= \pm z_5^3$, which is equivalent to $\Delta(\varphi)=J^2(\varphi) - z_5^6=0$.
\eep
Under the isomorphism $\Psi$ in (\ref{eq:Psi_isom}), let $\Psi_L$ denote the bilinear form associated with a line $L$.
\begin{cor} \label{Psi_deg}
     A line $L$ of $\bP(V_3)$ lies on an osculating plane of $C$ if and only if the bilinear form $\Psi_L$ is degenerate.
\end{cor}
\bep The bilinear form $\Psi_L$ is degenerate if and only if the matrix $\det(M_z)=0$ where $M_z$ is the Gram matrix  of $\Psi_L$  with respect to the basis $\mB_2$ of $V_2$.
Expanding  $\det(M_z)$ in powers of $z_5$ using \eqref{eq:Mz}, it is easy to calculate
\[\det(M_z) = 4(J(\varphi)- z_5^3).\]
By the above Lemma, we conclude that 
$\Psi_L$ is degenerate if and only if $L$  lies on an osculating plane of $C$. 
\eep
% \pm \sqrt{I(\varphi)}, \varphi)  \varphi \in V_4, I(\varphi) \text{ is a square in $F$}\} \]
We end this section by noting that the $PGL_2(F)$-orbits on $\mQ$ are determined in terms of the $PGL_2(F)$-orbits of quartic forms whose apolar invariant is a square in $F$. This is because the pair $(\varphi_{g \cdot L}, z_5(g \cdot L))$ with $z_5(g \cdot L)^2=I(\varphi_{g \cdot L})$ is (at the projective level) equal to $(g \cdot \varphi_L, \det(g)^{-2} z_5)$ where $z_5^2=I(\varphi_L)$. 
Each orbit of quartic forms, lifts to either a single orbit or two distinct orbits on $\mQ$. In the case $F = \bF_q$, the orbits of $PGL_2(q)$ on $\mQ$ will be described in \S \ref{S3}.

\section{$G$-orbits on $\bP(V_4)$ and $\mQ$} \label{S3}
\subsection{$G$-orbits of binary quartic forms} \hfill \\ Let $\mF_\Delta$ denote the set of binary quartic forms with non-zero discriminant. As mentioned in $\S \ref{S2}$, there can be more than one $G$-orbits of quartic forms over $\bF_q$ which have the same value of the $\jmath$-invariant. For example, the quartic forms $XY(X^2-Y^2)$ and $XY(X^2-\ep Y^2)$, where $\ep$ is a non-square in $\bF_q$, both have $1728$ as their $\jmath$-invariant, but they are in different orbits, because the first form splits over $\bF_q$ whereas the second form does not split completely over $\bF_q$. A finer invariant is needed to classify  the $G$-orbits on $\mF_\Delta$. This finer invariant is an equivalence class of `restricted cross-ratios' which we briefly recall (for details see \cite[\S 5]{KPP}). The set of quartic forms with non-zero discriminant,
naturally decomposes into parts  $\mF_4 \cup \mF_2 \cup \mF_1 \cup \mF_4' \cup \mF_2'$, where the set $\mF_i$ for $i \in \{1, 2, 4\}$ consists of those forms which have exactly $i$ linear factors over $\bF_q$. The set $\mF_2'$ consists of irreducible quartic forms, and the set $\mF_4'$ consists of forms which are a product of two distinct irreducible quadratic forms. The restricted cross ratio of a quartic form in 
\begin{enumerate}
    \item[(i)]  $\mF_4, \mF_4'$ is an element of \[\tilde \mN_4=\bF_q\setminus\{0,1\}.\]
    \item[(ii)]  $\mF_2, \mF_2'$ is an element of 
    \[ \tilde \mN_2=\{\lambda \in \bF_q^2\setminus\{1\} : \lambda^{q+1}=1\}. \]
    \item [(iii)]  $\mF_1$ is an element of 
    \[ \tilde \mN_1= \{\lambda \in \bF_q^3: \lambda^{q+1}  - \lambda^q +1=0\}. \]
\end{enumerate}
Let $H_4 \subset G$ denote the anharmonic group consisting of the transformations 
\[H_4= \{t \mapsto t, t^{-1}, 1-t, 1-t^{-1}, 1/(1-t), 1/(1-t^{-1})\}.\]
It is isomorphic to the symmetric group $S_3$, and is generated by the involution $t \mapsto t^{-1}$ and the order $3$ element $t \mapsto 1/(1-t)$. Let $H_2$ be the subgroup of $H_4$ generated by the involution $t \mapsto t^{-1}$. Also, let $H_1$ denote the trivial subgroup of $H_4$. The $G$-orbits in each of these five parts are classified by equivalence classes of restricted cross ratios: 
\begin{enumerate}
     \item $G$-orbits in $\mF_4 \;  \leftrightarrow \; \tilde \mN_4/H_4$.
     \item $G$-orbits in $\mF_4' \;  \leftrightarrow \; \tilde \mN_4/H_2$.
     \item $G$-orbits in $\mF_2 \;  \leftrightarrow \; \tilde \mN_2/H_2$.
     \item $G$-orbits in $\mF_2' \;  \leftrightarrow \; \tilde \mN_2/H_2$.
     \item $G$-orbits in $\mF_1 \;  \leftrightarrow \; 
     \tilde \mN_1/H_1$.
\end{enumerate}
The function $\jmath:\overline{\bF}_q\setminus\{0,1\} \to \overline{\bF}_q$ defined by 
\[  1-\tfrac{1728}{\jmath(\lambda)} = \frac{ (\lambda^2 - \lambda+1)^3}
{\left((\lambda+1)(\lambda-2)(\lambda-\tfrac{1}{2}) \right)^2},\]
classifies the $H_4$-orbits on $\overline{\bF}_q\setminus\{0,1\}$: $\jmath(\lambda) = \jmath(\lambda')$ if and only if $\lambda' \in H_4 \cdot \lambda$. Each $H_4$-orbit on  $\overline{\bF}_q\setminus\{0,1\}$ 
has size $6$  with two exceptions:
\[   \jmath^{-1}(0) = H_4 \cdot (-\omega)=\{-\omega, -\omega^2\}, \quad  \jmath^{-1}(1728) = H_4 \cdot (-1) =  \{-1,2,1/2\},\]
where $\omega$ is a primitive cube root of unity.
If $\lambda$ denotes a restricted cross-ratio corresponding to   an orbit $G \cdot \varphi(X,Y)$ in $\mF_\Delta$,  we define the quantity $\jmath(\mO)$ to be $\jmath(\lambda)$. This quantity also equals $\jmath(\varphi)$ as defined in \eqref{eq:jvarphi}, and it also equals the $\jmath$-invariant of the set of $4$ points $ \{(s_i,t_i) : 1 \leq i \leq 4\}$ of the projective line over $\overline{\bF}_q$.\\

For $i\in \{1, 2, 3\}$, let $\mN_i$ be the subset of $\tilde \mN_i$ defined by 
\[\mN_i = \tilde\mN_i \setminus\{-1,1/2,2,-\omega, -\omega^2\}.\]
The set $\bF_q\setminus\{0,1\}$ can be partitioned as 
\[ \bF_q\setminus\{0,1\} = J_4 \cup J_2 \cup J_1, \qquad J_i = \jmath(\mN_i).\]
The sets $J_i$ have sizes:
\[  |J_4|=\tfrac{(q-6-\mu)}{6}, \quad |J_2|=\tfrac{q-2+\mu}{2}, \quad  |J_1|=\tfrac{q-\mu}{3}.  \]
For $i \in \{1, 2, 3\}$, for each $r \in J_i$, there is one $G$-orbit in $\mF_i$.  For each $r \in J_4$ there  are $3$ orbits in $\mF_4'$, and  for each $r \in J_2$, there is one orbit in  $\mF_2'$. 
Thus, there are a total of $4|J_4| + 2 |J_2| +|J_1|=2q-6$ orbits $\mO$ in $\mF_\Delta$ with $\jmath(\mO) \in \bF_q \setminus\{0, 1728\}$. 
There are $5$ orbits in $\mF_\Delta$ with $\jmath(\mO)=1728$, of which there are two in $\mF_4'$ and one each in $\mF_4, \mF_2$ and  $\mF_2'$. There are  $3+\mu$ orbits in $\mF_\Delta$ with $\jmath(\mO)=0$, of which there are $(1+\mu)/2$ each in $\mF_4$ and $\mF_4'$, $(1-\mu)/2$ each in $\mF_2$ and $\mF_2'$, and $(1+\mu)$  in $\mF_1$. The sizes and  representative quartic forms for all the $(2q+2+\mu)$ orbits in $\mF_\Delta$ can be found in \cite[Table 3]{KPP}.

\subsection{$G$-orbits on $\mL$} \hfill \\
We recall from \S \ref{S2}, that 
under the $G$-equivariant isomorphism 
 $\Phi: \bP(\wedge^2 V_3) \to  \bP\left( V_4 \,  \oplus \,  (\wedge^2 V)^{\otimes 2}\right)$ 
of \eqref{eq:Phi_isom}, the points of the Klein quadric $\mQ$ have image 
\[ \Phi(\mQ) = \{ (\varphi,\pm \sqrt{I(\varphi)}) \colon \varphi \in V_4, I(\varphi) \text{ is a square in $\bF_q$}\}. \]
More precisely, the following result was proved in \cite{KPP}:
%let $\mF^+_\Delta$ denote the subset of $PG(V_4)$ consisting of quartic forms $\varphi$ whose apolar invariant $I(\varphi)$ is a square in $\bF_q$.  The following result was proved in \cite{KPP}:
\begin{thm} \cite[\S 3]{KPP}\label{lem_pi}
Let $\mF^+$ denote the subset of $PG(V_4)$ consisting of forms $\varphi$ with $I(\varphi)$ a square in $\bF_q$. There is a $PGL_2(q)$-equivariant $2$-sheeted covering map $\pi:\mQ \to \mF^+$, where for $\varphi \in \mF^+$ given by  
\[\varphi=z_0 Y^4-4 z_1 Y^3X+6 z_2 Y^2X^2 -4 z_3YX^3 + z_4X^4, \] 
the inverse image $\pi^{-1}(\varphi)$ consists of the  lines $\{L, L^{\perp}\}$ whose  coordinates $(z_0, \dots, z_5)$ satisfy  $z_5= \pm \sqrt{I(\varphi)}$. 

The following conditions are equivalent for a  pencil $L$ of $PG(V_3)$ with $\varphi = \pi(L)$
\begin{enumerate}
\item[i)] the pencil $L$  contains a form divisible by $(Xt-Ys)^2$,
\item[ii)] $\varphi(s,t)=0$.
\end{enumerate}
\end{thm}

A line $L$ is generic if and only if $\Delta(\varphi_L) \neq 0$.
For each $G$-orbit $\mO$ in $\mF^+_\Delta$, the set $\pi^{-1}(\mO)$ is either a single orbit $\fO  = \fO^\perp$ or a pair of orbits $\fO \cup \fO^\perp$ in $\mL$.  
For an orbit $\mO$ in $\mF_\Delta$, if $\jmath(\mO) \notin \{0, 1728\}$ then $\mO$ lifts to a pair of distinct orbits $\fO$ and $\fO^\perp$. 
If $\jmath(\mO) \in \{0, 1728\}$ then $\mO$ lifts to single self-dual  orbit $\fO=\fO^\perp$, with one exception: if $q \equiv \pm 1 \mod 12$, then of the two orbits in $\mF_4' \cap \mF^+_\Delta$ with $\jmath(\mO)=1728$, represented by $H_2 \cdot (-1)$ and $H_2 \cdot 2$, the orbit represented by $H_2 \cdot 2$ lifts to distinct orbits $\fO$ and $\fO^\perp$. The sizes and representative generators of all the 
$(2q-3+\mu)$ orbits of generic lines can be found in \cite[Table 4]{KPP}.

% Let $\mF^+_\Delta \subset \mF$ denote the subset of quartic forms $\varphi$ for which the apolar invariant $I(f)$ is a square in $\bF_q$. We recall from Lemma \ref{lem_pi} that $\pi:\mQ \to \mF^+_\Delta$ is a $G$-equivariant  $2$-sheeted covering branched over the forms with $I(f)=0$. For $f \in \mF^+_\Delta$ the set $\pi^{-1}(f)$ is of the form  $\{L , L^{\perp}\}$ where $L^{\perp}$ is the polar dual of $L$.   If $\mO = G \cdot f$ with $I(f)=0$, then $\pi^{-1}(\mO)$ forms a self dual orbit $\fO$ of size $|\mO|$. If  $\mO = G \cdot f$ with $I(f) \neq 0$, then $\pi^{-1}(\mO)$ either forms a single self dual orbit $\fO$ of size $2 |\mO|$, or a pair of orbits $\fO, \fO^{\perp}$ each of size $|\mO|$. 

\begin{prop}\cite[Proposition 6.2]{KPP} \label{J+prop} \hfill  \begin{enumerate}
\item There are $(3+\mu)$ orbits with $\jmath(\fO)=0$. 
 \begin{enumerate}
 	 \item If $\mu=-1$ both the orbits have size $|G|/2$. 
	  \item If $\mu=1$ then of the $4$ orbits, there are two orbits of size $|G|/3$ and one   orbit each of size $|G|/4$ and $|G|/12$.
 \end{enumerate}
\item   The number of orbits with $\jmath(\fO)=1728$ is
    \begin{enumerate}
 	 \item  $4$ if $q\equiv \pm 1 \mod 12$ all of which have size $|G|/4$.
	  \item $2$ if $q\equiv \pm 5 \mod 12$ both of which have size $|G|/2$.
    \end{enumerate}
\item  if $\jmath(G \cdot f) \neq 0, 1728$  and $\jmath(f) \in J_i$ for $i=1,2,4$ then 
$G \cdot f \in \mF^+_\Delta$  if and only if $\jmath(f) \in J_i^+$ where 
 \[ J_i^+ = \{r \in J_i : r/(r-1728) \text{ is a square in $\bF_q$}\}. \]
 The sets $J_i^+$ have sizes 
 \begin{enumerate}
\item $|J_1^+|=|J_1|/2 =(q-\mu)/6$ where $q \equiv \mu \mod 3$  and $\mu \in \{ \pm 1\}$.
\item $|J_4^+|= (q-r)/12$ where  $q \equiv r \mod 12$ and $r \in \{5,7,11,13\}$.
\item $|J_2^+|=\begin{cases} (q-1)/4 &\text{ if $q \equiv 1 \mod 12$,}\\
 (q-3)/4 &\text{ if $q \equiv 7 \mod 12$,}\\
   (q-5)/4 &\text{ if $q \equiv 5 \mod 12$,}\\
 (q-3)/4  &\text{ if $q \equiv 11 \mod 12$.}\end{cases}$
\end{enumerate}
\end{enumerate}
\end{prop}

% A  line $L$ is generic if and only if the quartic form $\varphi_L(X,Y)$ has $4$ distinct factors $\{ (Xt_i-Ys_i): i=1 \dots 4\}$.  The set of such quartic forms has been decomposed  (see \cite[Definition 5.1]{KPP}) as $\mF_1 \cup \mF_2 \cup \mF_2' \cup \mF_4 \cup \mF_4'$ where $\mF_i$ for $i=1,2,4$ consists of quartic forms with exactly $i$ factors defined over $\bF_q$. The set $\mF_2'$ consists of forms with all $4$ distinct factors defined over $\bF_{q^2}$ but not over $\bF_q$. The set $\mF_4'$ consists of forms with all $4$ distinct factors defined over $GF(q^4)$ but not over $\bF_{q^2}$.
Let $\jmath(\varphi)$ denote the $\jmath$-invariant of the $4$ roots $\{ (s_i,t_i) : i=1  \dots  4\}$ of $\varphi(X,Y)\in \mF_\Delta$ (see\cite[\S 4]{KPP}) and let $\jmath(\fO)=\jmath(\varphi)$ for any representative $\varphi$ of $\pi(\fO)$. The $(2q-3+\mu)$ orbits of generic lines corresponding to  $\jmath(\fO) =0, 1728$ and $\jmath(\fO) \in \bF_q \setminus \{0, 1728\}$ and their sizes were determined in the following table \cite[Table 2]{KPP}:\\

\noindent \begin{tabular}{c| *{5}{c}}
$\jmath(\fO)$   & $|G|$ & $\frac{|G|}{2}$ & $\frac{|G|}{3}$ & $\frac{|G|}{4}$ & $\frac{|G|}{12}$     \\
&&&&&\\  \hline \\
 $\jmath \neq 0, 1728$   & $2|J_1^+|$ & $4|J_2^+|$ & $0$  & $8|J_4^+|$ &  $0$   \\ 

$\jmath=1728$ & $0$ &  \small $\begin{cases} 0 &\text{ if $q \equiv \pm 1 \mod 12$} \\
2 &\text{ if $q \equiv \pm 5 \mod 12$} \end{cases}$  & $0$ & 
 \small $\begin{cases} 4 & \text{ if $q \equiv \pm 1 \mod 12$} \\
0 &\text{ if $q \equiv \pm 5 \mod 12$} \end{cases}$ & $0$   \\

$\jmath=0$  & $0$ & $(1-\mu)$ & $(1+\mu)$ & $\tfrac{1+\mu}{2}$ &  $\tfrac{1+\mu}{2}$  
  \\ [1ex]
\hline
total & $\tfrac{q-\mu}{3}$ & $q-1$ &  $(1+\mu)$ & $\tfrac{2q-10 -(1+\mu)/2}{3}$ &  $\tfrac{1+\mu}{2}$.
  \\ [1ex]
\end{tabular}

\section{The discriminant quartic form $D_L$ associated to a generic line $L$} \label{S4}

The $G$-orbit classification of points of $\bP(V_3)$ is given in \cite[Corollary 5]{Hirschfeld3}. We record these orbits in the next lemma.
\begin{lem}\label{cubic} 
The projective space of binary cubic forms over $\bF_q$ of size $q^3+q^2+q+1$ can be decomposed into the following five $G$-orbits.
\begin{enumerate}
    \item $G \cdot X^3$ of size $(q+1)$ (corresponding to the points of $C(\bF_q)$).
    \item $G \cdot X^2Y$ of size $q(q+1)$ (corresponding to the points not on $C(\bF_q)$ but on some tangent line of $C(\bF_q)$).
    \item $G \cdot XY(X-Y)$ of size $(q^3-q)/6$ (corresponding to the intersection points of the osculating planes at three distinct points of $C(\bF_q)$).
    \item $G \cdot X(X^2-\epsilon Y^2)$, where $\ep$ is a nonsquare in $\bF_q$, of size $q(q^2-1)/2$ (corresponding to the intersection points of the osculating planes to $C$ at $P,Q,R$, where $P$ is a point of $C(\bF_q)$ and $Q,R$ are two Galois conjugate points of $C(\bF_{q^2})$).
    \item $G \cdot (X-\theta Y)(X-\phi(\theta) Y)(X-\phi^2(\theta) Y)$, where $\theta \in \bF_{q^3}\setminus \bF_{q^2}$, of size $(q^3-q)/3$ (corresponding to the intersection points of the osculating planes to $C$ at $P,Q,R$, where $P,Q,R$ are three Galois conjugate points of $C(\bF_{q^3})$).
\end{enumerate}
\end{lem}

 We recall that the osculating plane $O_{(s,t)}$ to $C$ at a point $(Xt-Ys)^3$ consists of all elements of $PG(V_3)$ which are divisible by $(Xt-Ys)$. We use the notation $\bP(V_m \otimes \overline{\bF_q})$ for the projective space of degree $m$ binary  forms $f(X,Y)$ over an algebraic closure $\overline{\bF_q}$. For a line of $PG(V_3)$, let $\bar{L}$ denote the line of  $\bP(V_3 \otimes \overline{\bF_q})$ consisting of the $\overline{\bF_q}$-points of $L$. If $L$ is not contained in an osculating plane of $C$, then $\bar{L}$ intersects each osculating plane of $C$ in a unique point. In other words, for each $(Xt-Ys) \in \bP(V \otimes \overline{\bF_q})$, the pencil $\bar{L}$ contains a unique cubic form  $(Xt-Ys) \cdot h_{(s,t)}^L(X,Y)$  in  $\bP(V_3 \otimes \overline{\bF_q})$. We will now determine the quadratic form $h_{(s,t)}^L(X,Y)$ in terms of the coordinates $(z_0, \dots, z_5)$ of $L$. The bilinear form $\Psi_L$ on $V_2\otimes \overline{\bF_q}$ is non-degenerate, and hence gives a polarity on $\bP(V_2\otimes \overline{\bF_q})$. The polar dual of a point $f \in \bP(V_2\otimes \overline{\bF_q})$ with respect to this polarity will be denoted $f^{\perp_{\Psi_L}}$. The polar dual of $f$ with respect to the polarity given by $\Omega_2$ will be denoted $f^{\perp_{\Omega_2}}$.
\begin{prop}\label{hLXY}
If $L$ does not lie in any osculating plane of $C$, then  for each $(Xt-Ys) \in \bP(V \otimes \overline{\bF_q})$, the quadratic form  $h_{(s,t)}^L(X,Y)$  
is the unique element of 
$\bP(V_2\otimes \overline{\bF_q})$ such that 
\[ (h_{(s,t)}^L)^{\perp_{\Omega_2}} = \left( (Xt-Ys)^2 \right)^{\perp_{\Psi_L}}.\]
  In coordinates
\[h_{(s,t)}^L(X,Y) =\bbsm X^2 &  XY & Y^2 \besm M_z \bbsm s^2 \\ st \\ t^2 \besm.\]
\end{prop}
\begin{proof}
If $h_{(s,t)}^L = X^2 \alpha_1(s,t) +\alpha_2(s,t) XY + \alpha_3(s,t) Y^2$,
then 
\[ (h_{(s,t)}^L)^{\perp_{\Omega_2}}=\{aY^2 -2b XY +cX^2 \colon 
\bbsm a & b & c \besm \bbsm \alpha_1(s,t) \\ \alpha_2(s,t) \\ \alpha_3(s,t) \besm = 0\}.\] 
On the other hand
\[ \left( (Xt-Ys)^2 \right)^{\perp_{\Psi_L}}= \{ a Y^2 -2bXY + cX^2 \colon \bbsm a & b & c \besm M_z \bbsm s^2 \\ st \\t^2 \besm=0\}. 
\]
Thus, the condition that $(h_{(s,t)}^L)^{\perp_{\Omega_2}} = \left( (Xt-Ys)^2 \right)^{\perp_{\Psi_L}}$ is equivalent to 
\[h_{(s,t)}^L(X,Y) =\bbsm X^2 &  XY & Y^2 \besm M_z \bbsm s^2 \\ st \\ t^2 \besm.\]
We represent  $L$  as a pencil $(\mu,\nu) \mapsto \mu u(X,Y) +\nu v(X,Y)$ for $(\mu,\nu) \in \bP^1(\overline{\bF_q})$. Since  $L$
is not contained in any osculating plane of $C$, and since $O_{(s,t)}$  consists of all cubic forms divisible by $(Xt-Ys)$,
the quantities $u(s,t)$ and $v(s,t)$ do not simultaneously vanish. Therefore,  the unique element of  $\bar{L}$ in $O_{(s,t)}$
is $(Xt-Ys) h_{(s,t)}^L(X,Y)$ where 
\[ h_{(s,t)}^L(X,Y)=\frac{v(s,t) u(X,Y) -u(s,t) v(X,Y)}{Xt-Ys}.\]
Writing $h_{(s,t)}^L(X,Y) = X^2 \alpha_1(s,t) +\alpha_2(s,t) XY + \alpha_3(s,t) Y^2$, it is clear that 
$\alpha_i(s,t)$ only depend on the 
Pl\"ucker coordinates $u_iv_j - u_jv_i=p_{ij}$ of $L$, where $(u_0, \dots, u_3)$  and $(v_0, \dots, v_3)$ are the coordinates of $u(X,Y)$ and $v(X,Y)$ with respect to the basis $\mathcal B_3$. In terms of the coordinates $(z_0,\dots,z_5)$ with respect to the basis $\mathcal E_5$ we have 
\[ \bbm \alpha_1(s,t)\\ \alpha_2(s,t)\\ \alpha_3(s,t) \bem = M_z \bbm s^2 \\ st \\t^2 \bem. \]
\end{proof}
\begin{lem} \label{h_lem} Let $L$ be   line  which is not contained in any osculating plane of $C$. 
For $j \in \{1,2\}$ let 
\[ a_j(L)=\{(s,t) \in PG(1,q) \colon  (Xt-Ys)h^{L}_{(s,t)}(X,Y) \text{ has  $j$ distinct linear factors}\}. \]
For $i \in \{1,3\}$ let:
\[ a_{3,i}(L) = \{(s,t) \in PG(1,q) \colon  (Xt-Ys)h^{L}_{(s,t)}  \text{ has $3$ distinct linear factors of which $i$ are over $\bF_q$}\}.\]
Let $\mS$ denote the set of $(q+1)$ points on $L$. Then we have 
\begin{enumerate}
\item $|\mS \cap O_1|=|a_1(L)|$,
\item $|\mS \cap O_2|= |a_2(L)|/2$,
\item $ |\mS \cap O_3|= |a_{3,3}(L)|/3$,
\item $|\mS \cap O_4| = |a_{3,1}(L)|$,
\item $|\mS \cap O_5|= (q+1)- (|a_1(L)|+  |\frac{|a_2(L)|}{2}+ \frac{|a_{3,3}(L)|}{3} +|a_{3,1}(L)|)$. 
\end{enumerate}
\end{lem}
\begin{proof}
Given distinct points $(s_1,t_1), (s_2,t_2) \in PG(1,q)$ we note that \[ (Xt_1-Ys_1)h^{L}_{(s_1,t_1)}=(Xt_2-Ys_2)h^{L}_{(s_2,t_2)}, \]  if and only if both cubic forms equal
\[ (Xt_1-Ys_1)(Xt_2-Ys_2)(Xt_3-Ys_3), \]  
for some $(s_3,t_3) \in PG(1,q)$. Thus $|a_{3,3}(L)|=3 |\mS \cap O_3|$ and $|a_2(L)|=2 |\mS \cap O_2|$.   Similarly, we get  bijections $\mS \cap O_4 \to a_{3,1}(L)$ and $\mS \cap O_1 \to a_{1}(L)$ that take a cubic form $f \in \mS$ to $(s,t) \in PG(1,q)$  where $(Xt-Ys)$ is the unique  factor of $f$ over $\bF_q$.  The remaining quantity $|\mS \cap O_5|$ is $|\mS| - \sum_{i=1}^4 |\mS \cap O_i|$.
\end{proof}
\begin{definition} \label{h_def}
Let $L$ be  a line which is not contained in any osculating plane of $C$. 
The discriminant of the quadratic form $h_{(s,t)}^L(X,Y)$ is given by $4 D_L(s,t)$ where $D_L(X,Y) \in V_4$ is:
 \[ D_L(X,Y)  = -\tfrac{1}{2} \bbsm X^2 & XY & Y^2 \besm M_z  A_2^{-1} M_z \bbsm X^2 \\ XY \\ Y^2 \besm. \]
 We can expand this as 
 \begin{multline} \label{eq:D_L_long} 
 D_L(X,Y)= 
 -z_5\varphi_L(X,Y)
 +(z_1^2-z_0z_2)Y^4 + 2(z_0z_3-z_1z_2)Y^3X \\  -(z_0z_4+2z_1z_3-3z_2^2)X^2Y^2+2 (z_1z_4-z_2z_3)YX^3 +(z_3^2 -z_2z_4)X^4.
 \end{multline}
\end{definition}
Let 
\beq \label{eq:nu_def} \nu_L = \#\{(s,t) \in PG(1,q) \colon D_L(s,t) \text{ is a non-zero square in $\bF_q$}\}.\eeq
Let 
\beq \label{eq:eta_def} \eta_L:=\begin{cases} i &\text{if $\varphi_L \in \mF_i$ for $i=1,2,4$}\\
 0 &\text{if $\varphi_L$ is in $\mF_2'$ or $\mF_4'$}. \end{cases}\eeq
\begin{prop}\label{gen_prop}
Let $L$ be a generic line of $\bP(V_3)$. Let $D_L(X,Y)$, $\eta_L$ and $\nu_L$ be as defined above. The quartic form $D_L(X,Y)$  has non zero discriminant, and is the same type $\mF_1, \mF_2, \mF_4, \mF_2', \mF_4'$ as $\varphi_L$.
We have 
\begin{enumerate}
\item $|\mS \cap O_1|=0$,
\item $|\mS \cap O_2|= \eta_L$,
\item $ |\mS \cap O_3|= (\nu_L-\eta_L)/3$,
\item $|\mS \cap O_4| = q+1-\nu_L-\eta_L$,
\item $|\mS \cap O_5|= (2 \nu_L+\eta_L)/3$. 
\end{enumerate}
\end{prop}
\bep 
We note that 
\begin{align*} 
h^L_{(s,t)}(s,t) &= \bbsm s^2 & st & t^2 \besm M_z \bbsm s^2 \\st \\ t^2 \besm\\
&= z_5 \bbsm s^2 & st & t^2 \besm A_2 \bbsm s^2 \\ st \\ t^2 \besm + \bbsm s^2 & st & t^2 \besm M_{\varphi_L} \bbsm s^2 \\st \\t^2 \besm\\
&=\varphi_L(s,t).
\end{align*}
Therefore, 
\begin{align*}
h^L_{(s,t)}(X,Y)=(Xt'-Ys')^2 \Leftrightarrow  h^L_{(s',t')}(X,Y)=(Xt'-Ys')(Xt-Ys) \\
\Leftrightarrow h^L_{(s',t')}(s',t')=0 \Leftrightarrow \varphi_L(s',t')=0. \end{align*}
Let $\{ (Xt_i'-Ys_i') : 1 \leq i \leq 4\}$ be the four distinct linear factors of $\varphi_L$, and let $F_i \supset \bF_q$ be the smallest extension field of $\bF_q$ such that the linear form $(Xt_i'-Ys_i') \in \bP(V \otimes \overline{\bF_q})$ is defined over $F_i$. It follows that for each $1 \leq i \leq 4$, we have 
\[ (Xt_i'-Ys_i')h^L_{(s_i',t_i')}=(Xt_i'-Ys_i')^2(Xt_i-Ys_i),\]
for some linear forms $(Xt_i-Ys_i)$ defined over $F_i$.  Since $h^L_{(s_i,t_i)}=(Xt_i'-Ys_i')^2$, we note that $(Xt_i'-Ys_i')$ is a factor of $D_L(X,Y)$. Since $L$ does not meet $C$,  the cubic form $(Xt_i-Ys_i) h^L_{(s_i,t_i)} \neq (Xt_i-Ys_i)^3$ and hence $(Xt_i-Ys_i) \neq (Xt_i'-Ys_i')$ in $\bP(V \otimes \overline{\bF_q})$. Moreover, 
 Since $(Xt_i-Ys_i) h^L_{(s_i,t_i)}= (X t_i'-Ys_i')h^L_{(s_i',t_i')}$, it follows that $(Xt_i-Ys_i)=h^L_{(s_i,t_i)}/(Xt_i'-Ys_i')$ is defined over $F_i$. If 
  $(Xt_i-Ys_i)$ is defined over a subfield $K$ of $F_i$, then $(Xt_i'-Ys_i')=\tfrac{h^L_{(s_i,t_i)}}{(Xt_i-Ys_i)}$ is also defined over $K$. This shows that $F_i \supset \bF_q$ is the smallest field over which $(Xt_i-Ys_i)$ is defined.
 We also claim that the four forms $(Xt_i-Ys_i)$ are distinct in $\bP(V \otimes \overline{\bF_q})$: 
  if $(Xt_1-Ys_1)=(Xt_2-Ys_2)$ then 
\[ (Xt_1'-Ys_1')=\tfrac{h^L_{(s_1,t_1)}}{Xt_1-Ys_1}=\tfrac{h^L_{(s_2,t_2)}}{Xt_2-Ys_2}=(Xt_2'-Ys_2'),\]
which is not the case.   We conclude that $D_L(X,Y)$ does not have repeated roots, and it is in the same part of the decomposition $\mF_1 \cup \mF_2 \cup \mF_4 \cup\mF_2' \cup \mF_4'$ as $\varphi_L$. 
In particular, $D_L(X,Y)$ has the same number  $\eta_L$ of linear factors over $\bF_q$  as $\varphi_L$. 
Let 
\[ \tilde\nu_L = \#\{(s,t) \in PG(1,q) \colon D_L(s,t) \text{ is a non-square in $\bF_q$}\}.\]
Clearly  $\tilde \nu_L + \nu_L + \eta_L=|PG(1,q)|=q+1$. In order to prove the assertions $(1)-(5)$ in the Proposition statement, it suffices to show that the quantities $|a_j(L)|, |a_{3,i}(L)|$ of Lemma  \ref{h_lem} are \[|a_1(L)|=0,\; |a_2(L)|=2 \eta_L, \;|a_{3,3}(L)|=\nu_L-\eta_L, \;|a_{3,1}(L)|=\tilde\nu_L.\]
Since $L$ does not intersect $C$, the cubic form $(Xt-Ys)h^L_{(s,t)} \neq (Xt-Ys)^3$. This shows (i) that $|a_1(L)|=0$, and (ii) $(Xt-Ys)h^L_{(s,t)}$ has discriminant zero if and only if it has $2$ distinct linear factors over $\bF_q$ (namely $(Xt_i-Ys_i), (Xt_i'-Ys_i')$ where $(Xt_i'-Ys_i')$ is  a linear factor of $\varphi_L$). Hence, we get $|a_2(L)|/2= \eta_L$.
We have $a_{3,1}(L)$ consists of those $(s,t) \in PG(1,q)$ such that $  h^L_{(s,t)}$ is irreducible over $\bF_q$, i.e. $|a_{3,1}(L)|=\tilde \nu_L$.  Also $a_{3,3}(L)$ consists of those $(s,t) \in PG(1,q)$ such that $h^L_{(s,t)}$ has two distinct linear factors, none of which is $(Xt-Ys)$ itself. This means 
$|a_{3,3}(L)|=\nu_L- \eta_L$.
\eep

In the next result, we determine the invariants $I(D_L), J(D_L)$ ad $\jmath(D_L)$ of the quartic form $D_L$.
\begin{prop} \label{I*J*prop} Let $L$ be a generic line of $\bP(V_3)$ represented by the pair $(\varphi_L, z_5(L))$. We denote $z_5(L)$ as $\sqrt{I_\varphi}$.
\begin{align} \label{eq:I*J*}  I(D_L) &= J_\varphi \sqrt{I_\varphi} + \tfrac{5}{4}I_\varphi^2, \\
 \nonumber J(D_L)&= \tfrac{-1}{8}( 11  I_\varphi^3 + 2 J_\varphi^2 + 14 J_\varphi \sqrt{I_\varphi}^3 ),\\
 \nonumber  1 - \tfrac{1728}{\jmath(D_L)} &= \frac{(11+2r^2+14r)^2}{(4r+5)^3}, \quad r = \sqrt{ 1 - \tfrac{1728}{\jmath(\varphi)}}=\tfrac{J(\varphi)}{{\sqrt I_\varphi}^3}.
\end{align}
\end{prop}
\bep  Let $F \supset \bF_q$ be an extension  field of $\bF_q$.  Let $(z_0, \dots, z_5)$ denote the coordinates of $L$ with respect to the basis $\mE_5$ of $\wedge^2 V_3$. For $g \in GL_2(F)$, we first show that 
\beq \label{eq:D_gL}\det(g)^{4} D_{g \cdot L} =  g \cdot D_L.\eeq
We have 
\[ g \cdot D_L=
 \tfrac{-1}{2} \bbsm X^2 & XY & Y^2 \besm A_2 g_2 A_2^{-1} M_z  A_2^{-1} M_z A_2^{-1} g_2^\top A_2^{-1} \bbsm X^2 \\ XY \\ Y^2 \besm. \]
Using 
$g_2 A_2^{-1} g_2^{\top} = \det(g)^{-2} A_2^{-1}$
(from \eqref{eq:Aj_invariance}), we get 
\[ g \cdot D_L=
 \tfrac{-\det(g)^{-4}}{2} \bbsm X^2 & XY & Y^2 \besm g_2^{-\top} M_z^t  A_2^{-1} M_z g_2^{-1}   \bbsm X^2 \\ XY \\ Y^2 \besm. \]
Using \eqref{eq:M_invariance}
 $M_{g \cdot z} = \det(g)^{-5} g_{2}^{-\top} M_z g_{2}^{-1}$ we get 
\[ g \cdot D_L=
 \tfrac{-\det(g)^{6}}{2} \bbsm X^2 & XY & Y^2 \besm  
 M_{\tilde g_5 z} g_2  A_2^{-1} g_2^{\top} M_{\tilde g_5 z}   \bbsm X^2 \\ XY \\ Y^2 \besm. \]
Using $g_2 A_2^{-1} g_2^{\top} = \det(g)^{-2} A_2^{-1}$ once again, we get 
 \[ g \cdot D_L=
 \tfrac{-\det(g)^{4}}{2} \bbsm X^2 & XY & Y^2 \besm  
 M_{\tilde g_5 z}  A_2^{-1}  M_{g z}    \bbsm X^2 \\ XY \\ Y^2 \besm, \]
which is $\det(g)^4 D_{g \cdot  L}$.\\
Using \eqref{eq:IJ_invariance}, we get
\beq \label{eq:IJD_L} I(D_L)= \det(g)^{12} I(D_{gL}), \quad J(D_L)=\det(g)^{18} J(D_{gL}).  \eeq

As shown in \cite[\S 4]{KPP}, 
then there exists $g \in GL_2(F)$, where $F$ is any extension field of $\bF_q$ over which  $\varphi$ has a linear factor,  with the property that $g \cdot \varphi = \det(g)^{-2} X R_\varphi(X,Y)$ where $R_\varphi(X,Y) = -4 Y^3 + 3 I_\varphi YX^2 -J_\varphi X^3$ is the cubic resolvent of the quartic form $\varphi(X,Y)$.  In particular, 
\beq \label{eq:g_5_z}\tilde g_5 z=\det(g)^{-3} (0,1,0,-3 I_\varphi/4,-J_\varphi,z_5)\eeq
Using \eqref{eq:D_L_long} we have:
 \[ \det(g)^6 D_{g L}(X,Y)= Y^4+
 4 z_5  XY^3  +\tfrac{3I_\varphi}{2}  X^2Y^2 - (2 J_\varphi +3 z_5 I_\varphi)YX^3 +(z_5 J_\varphi +\tfrac{9 I_\varphi^2}{16}) X^4 \]
Further using \eqref{eq:I} and \eqref{eq:J}, we get:
\[ \det(g)^{12} I(D_{gL})= z_5 J_\varphi +\tfrac{5}{4} I_\varphi^2, \quad  \det(g)^{18} J(D_{gL})=
 \tfrac{-1}{8}( 11  I_\varphi^3 + 2 J_\varphi^2 + 14 J_\varphi  z_5 I_\varphi).
 \]
 Using this in \eqref{eq:IJD_L}, we get the asserted values \eqref{eq:I*J*} for $I(D_L)$ and $J(D_L)$. Finally, the identity  $1 - 1728/\jmath(f) = J^2(f)/I^3(f)$ for a quartic form $f$, gives the third equation in \eqref{eq:I*J*}.
 \eep
\section{The Elliptic curve $E_L$ associated to a line $L$} \label{S5}
Let $L$ be a  generic line $L$ of $\bP(V_3)$. Let  $D_L(X,Y)$, $I(D_L), J(D_L)$, $\nu(L)$ and $\eta(L)$ be as defined above in \eqref{eq:D_L_long}, \eqref{eq:I*J*}, \eqref{eq:nu_def} and \eqref{eq:eta_def}. We recall that 
\[\nu(L)=\#\{(x,y) \in PG(1,q) :   D_L(x,y) \text{ is a  non-zero square in $\bF_q$}\},\]
and $\eta(L)$ is the number of linear factors of $\varphi_L$ over $\bF_q$. 
We also define the quantities
\begin{align} \label{eq:g2g3def}
\nonumber    g_2(L)&=3 I(D_L)=3z_5(L) J_\varphi  + \tfrac{15}{4}I_\varphi^2, \\
    g_3(L)&=J(D_L)=\tfrac{-1}{8}( 11  I_\varphi^3 + 2 J_\varphi^2 + 14 z_5(L) J_\varphi  I_\varphi )
\end{align}

In the next theorem we obtain an expression for  $\nu_L$ in terms of the number $\#E_L(\bF_q)$ of points over $\bF_q$ of  the  elliptic curve $E_L$ in $\bP^2$ given  by:
\beq \label{eq:E_L_def} E_L: \quad  T^2=4 S^3  -g_2(L) S -g_3(L). \eeq
\begin{thm} \label{thm_elliptic} 
We have    \[ \nu_L = (\# E_L(\bF_q)-\eta(L))/2. \] 
\end{thm}
\bep
We recall that for a non-zero quartic form $f \in V_4$ and $g =\bbsm a & b \\ c & d \besm \in GL_2(q)$, we have $g \cdot f = \det(g)^{-4} f(dX-bY, aY-cX)$. Therefore, $g^{-1}$ carries the set $\{(s,t) \in PG(1,q) \colon f(s,t) \text{ is a non-zero square}\}$ bijectively to the set  $\{(s,t) \in PG(1,q) \colon (g \cdot f)(s,t) \text{ is a non-zero square}\}$. In particular, both sets have the same size. 
Applying this to the quartic form $D_L$, we conclude  that for any $g \in GL_2(q)$, we have $\nu(L)=\#\{(x,y) \in PG(1,q) :   (g \cdot D_L)(x,y) \text{ is a  non-zero square in $\bF_q$}\}$.\\

We also recall that for any field $F \supset \bF_q$ over which a quartic form $f$ has a linear factor, there exists $g \in GL_2(F)$
 such that $g \cdot f = \det(g)^{-2} X(-4Y^3 +3I(f)  YX^2 -J(f) X^3)$.
In case $\varphi_L \in \mF_i$ for $i \in \{1, 2, 4\}$, we recall that $D_L$ has $\eta_L = i$ linear factors over $\bF_q$. 
Therefore,  there exists $g \in GL_2(q)$ such that 
\[ g \cdot D_L = \det(g)^{-2} X(-4Y^3 +3I(D_L)  YX^2 -J(D_L) X^3).\]
Since $(g \cdot D_L)(x,y)=0$ for $(x,y) =(0,1)$ we have in terms of $s = -y/x$
\[ \nu_L = \#\{s \in \bF_q \colon 4 s^3 -g_2(L) s - g_3(L) \text{ is a non-zero square in $\bF_q$.}\}\]
We also note that $4 s^3 -g_2(L) s - g_3(L)$ has $(\eta(L)-1)$ roots in $\bF_q$ because $X(-4Y^3 +3I(D_L)  YX^2 -J(D_L) X^3)$ has $\eta_L$ linear factors over $\bF_q$.
The number $\# E_L(\bF_q)$ of the points in $PG(2,q)$ of the elliptic curve $E_L$ elliptic curve is clearly 
\[2 \nu_L+ (\eta_L-1) + 1,\]
where the contribution $1$ comes from point at infinity, the term $(\eta_L-1)$ comes from the roots in $\bF_q$ of $(4 s^3 -g_2(L) s - g_3(L))$. Thus, we have shown that 
\[ \nu_L = (\#E_L(\bF_q) - \eta_L)/2.\]

We now turn to the case when $\varphi_L$ (and hence $D_L$) does not have a linear factor over $\bF_q$, i.e. $D_L \in \mF_2' \cup \mF_4'$. 
We write 
\[D_L(X,Y) = a_0 Y^4
-4a_1Y^3X+6a_2Y^2X^2-4a_3YX^3 +a_4X^4.
\]
  % In particular,  $D_L(x,y) \neq 0$ for all $(x,y) \neq (0,0) \in \bF_q \times \bF_q$. 
We consider  the  curve in $\mathbb  P^2$
\[ X^2W^2= D_L(X,Y).\]
Since  $D_L$ has no repeated factors, the only singularity of this curve is the point $(X,Y,W)=(0,0,1)$, which is a cusp singularity. We define $\mE_L$ to be a non-singular model of this  curve. The curve $\mE_L$ has genus $1$, and the singular point of the original curve corresponds to a pair of points of $\mE_L$ around which a  model of $\mE_L$ is:
\[  (\frac{WX}{Y^2})^2 = a_0 -4  a_1 (\frac{X}{Y}) + 6 a_2 (\frac{X}{Y})^2 -4 a_3(\frac{X}{Y})^3 + a_4 (\frac{X}{Y})^4. \]
The above mentioned pair of points is $(\tfrac{WX}{Y^2},\tfrac{X}{Y}) = (\pm \sqrt{a_0},0)$.
These $2$ points are defined over $\bF_q$ if and only if  $a_0=D_L(0,1)$ is a  square  in $\bF_q$. We also note the remaining points of $\mE_L(\bF_q)$ are 
\[ \{ (X,Y,W)=(1,y, \pm \sqrt{D_L(1,y)})   \colon D_L(1,y) \text{ is a non-zero square in $\bF_q$}\}.\]
(we recall that $D_L(1,y) \neq 0$ for all $y \in \bF_q)$.
Therefore,
\[ \nu_L=\# \mE_L(\bF_q)/2.\]
By the Hasse bound, $\# \mE_L(\bF_q) \geq (\sqrt q -1)^2 >0$ and hence $\nu_L >0$. By the definition of $\nu_L$, we conclude that there is a point $(x_0,y_0) \in PG(1,q)$ such that $D_L(x_0,y_0)$ is a non-zero square in $\bF_q$. Let $g \in SL_2(q)$ such that $g (x_0,y_0)=(0,1)$. Since $\nu_L$ is unaffected if we replace $D_L$ by $g \cdot D_L$,
we may assume $(x_0,y_0)=(0,1)$ or equivalently $D_L(0,1)=a_0$ is a non-zero square in $\bF_q$. 
Again replacing $D_L$ by $g \cdot D_L$ where $g=\bbsm 1 & 0\\-a_1/a_0 & 1\besm \in SL_2(q)$, we may assume
\[g \cdot D_L(X,Y) = a_0 Y^4  + 6 a_2 X^2Y^2 -4 a_3X^3Y + a_4 X^4.  \]
Since $g \in SL_2(q)$,  $I(g \cdot D_L) = \det(g)^{-4} I(D_L) = I(D_L)$ and $J(g \cdot D_L) = \det(g)^{-6} J(D_L) = J(D_L)$. Therefore,
\begin{align*}  g_2(L)&=3 I(D_L)=(a_0 a_4+3a_2^2),\\
g_3(L)& = J(D_L) =a_0a_2a_4-a_2^3-a_0a_3^2. \end{align*}
We take  $\mE_L$ to be a non-singular model of the projective plane curve
\[X^2W^2= a_0 Y^4  + 6 a_2 X^2Y^2 -4 a_3X^3Y + a_4 X^4.\]
We show that the curves $\mE_L$ and $E_L$ are isomorphic. Our proof closely follows  Theorem 2 in \S 10 of  Mordell's book \cite{Mordell}. 
Let $P_+$ and $P_-$ denote the points of $\mE_L$ with coordinates $(XW/Y^2, X/Y)$ being $ (\sqrt{a_0},0)$ and $(-\sqrt{a_0},0)$ respectively. Let $P_\infty$ denote the point at infinity of $E_L$ with coordinates $(1/T,S/T)=(0,0)$.  We define a map 
$\psi: \mE_L \to E_L$ as follows:
It is useful to rewrite the two curves (away from the points at infinity) as:
\begin{align} \label{eq:E_alt}
\nonumber \mE_L&: (Y^2+\tfrac{3a_2}{a_0}-\tfrac{W}{\sqrt a_0})(Y^2+\tfrac{3a_2}{a_0}+\tfrac{W}{\sqrt a_0})= \tfrac{9a_2^2}{a_0^2} + \tfrac{4 a_3}{a_0} Y -\tfrac{a_4}{a_0} \\
E_L&: (T - \sqrt a_0 a_3) (T + \sqrt a_0 a_3)=(S+a_2)( (2S-a_2)^2 - a_0a_4).
\end{align}
Let $Q_{\pm}$ denote the points $(S,T)=(-a_2, \pm \sqrt a_0 a_3)$  of $E_L$. (If $a_3=0$, $Q_{\pm}$ is a single point.)\\

The isomorphism $\psi: \mE_L \to E_L$ that we construct below satisfies $\psi(P_-)=Q_+$, $\psi(P_+)= P_\infty$  and $\psi$
maps $\mE_L \setminus \{P_+, P_-\}$ bijectively to $E_L \setminus \{P_{\infty}, Q_+\}$. First, we define $\psi:\mE_L \setminus \{P_-, P_+\} \to E_L \setminus\{ P_\infty, Q_+\}$ by 
\[ (S+a_2)=\tfrac{a_0}{2} (Y^2+\tfrac{3a_2}{a_0}+\tfrac{W}{\sqrt a_0}), \quad (T+\sqrt a_0 a_3)=a_0^{3/2} Y (Y^2+\tfrac{3a_2}{a_0}+\tfrac{W}{\sqrt a_0}). \]
It is readily checked that the indicated point does lie on $E_L$: plugging in the prescribed values of  $(S+a_2)$ and $(T+\sqrt a_0 a_3)$ in the equation \eqref{eq:E_alt} of $E_L$, and cancelling the common factor of $a_0^2(Y^2+\tfrac{3a_2}{a_0}+\tfrac{W}{\sqrt a_0})$, we get
\[ a_0Y^2(Y^2+\tfrac{3a_2}{a_0}+\tfrac{W}{\sqrt a_0}) - 2 a_3 Y = \tfrac{a_0}{2} (Y^2+\tfrac{W}{\sqrt a_0})^2- \tfrac{a_4}{2},\]
which is true upon using the equation $W^2=a_0Y^4+6a_2 Y^2-4 a_3 Y +a_4$. Next, we show that $\psi(P_-) = Q_+$.  
Near $P_-: (W/Y^2, 1/Y) = (-\sqrt a_0, 0)$, we have 
\begin{align*}
(Y^2+\tfrac{3a_2}{a_0}+\tfrac{W}{\sqrt a_0})_{|P_-} &=\left( \frac{ \tfrac{9a_2^2}{a_0^2Y^2} + \tfrac{4 a_3}{a_0 Y}  -\tfrac{a_4}{a_0Y^2} }{1+\tfrac{3a_2}{a_0 Y^2}-\tfrac{W}{\sqrt a_0 Y^2}}\right)_{|P_-}=\frac{0}{2}=0, \\
Y(Y^2+\tfrac{3a_2}{a_0}+\tfrac{W}{\sqrt a_0})_{|P_-} &=\left( \frac{ \tfrac{9a_2^2}{a_0^2Y} + \tfrac{4 a_3}{a_0}  -\tfrac{a_4}{a_0Y} }{1+\tfrac{3a_2}{a_0 Y^2}-\tfrac{W}{\sqrt a_0 Y^2}}\right)_{|P_-}=\frac{4a_3/a_0}{2}=\frac{2a_3}{a_0}.
\end{align*}
Therefore, $\psi(P_-)$ is the point $Q_+: (S,T) = (-a_2, \sqrt a_0 a_3)$.\\

We now show that $\psi(P_+)= P_\infty$. Near $P_+: (W/Y^2,1/Y)=(\sqrt a_0, 0)$, we can express $\psi$ in terms of the coordinates $(1/T,S/T)$ near $P_{\infty}$ as:
\[ (\tfrac{1}{T},\tfrac{S}{T})_{|P_+}=
\left(\frac{\tfrac{1}{Y^3}}{a_0^{3/2}+\tfrac{ a_0 W}{ Y^2}+\tfrac{2a_3 \sqrt a_0}{Y^3}}, \frac{\tfrac{1}{2Y}(\tfrac{a_2}{Y^2} + a_0 + \tfrac{\sqrt a_0 W}{Y^2})}{a_0^{3/2} +\tfrac{a_0 W}{Y^2}+ \tfrac{2 a_2 \sqrt a_0}{Y^2}}\right)_{|{P_+}}=(0,0).\]
Thus, $\psi(P_+)=P_\infty$.\\

We define a map $\psi':E_L \to \mE_L$  as follows: First we define $\psi': E_L \setminus\{P_\infty, Q_+\} \to \mE_L \setminus\{P_+, P_-\}$ by 
\[ 
(Y,W)=\tfrac{1}{\sqrt a_0} \left( \frac{T +\sqrt{a_0} a_3}{2(S+a_2)},  2S-a_2 -  (\tfrac{T+ \sqrt{a_0} a_3}{2(S+a_2)})^2 \right).  \]
 It is readily checked that this point  lies on $\mE_L$. Near $P_\infty$, we express $\psi'$ in terms of the coordinates $(1/T,S/T)$ around $P_\infty$ and the coordinates 
$(W/Y^2,1/Y)$ around $P_+$ by $(1/T,S/T) \mapsto$
\[ (\tfrac{W}{Y^2},\tfrac{1}{Y})=\sqrt a_0 \left(-1 + 2  \frac{(2S-a_2)^2 (S+a_2)}{T^2}  \frac{(1+\tfrac{a_2}{S})}{(1-\tfrac{a_2}{2S})(1+ \tfrac{\sqrt a_0 a_3}{T})^2}, 
 \frac{2(\tfrac{S}{T}+\tfrac{a_2}{T})}{1 +\tfrac{\sqrt{a_0} a_3}{T}} \right). \]
It is clear that $\tfrac{1}{Y}_{|P_\infty}=0$.
Writing the equation of $E_L$ as
\[ \tfrac{1}{S}=4 (\tfrac{S}{T})^2 -g_2 \tfrac{S}{T} \tfrac{1}{T^2} -g_3 \tfrac{1}{T^3},\]
we see that $\tfrac{1}{S}_{|P_\infty}=0$. Therefore, \[ \tfrac{W}{\sqrt a_0 Y^2}_{|P_\infty}  = -1 + 2  \frac{(2S-a_2)^2 (S+a_2)}{T^2}_{|P_\infty}. \]
Again writing the equation of $E_L$ as
\[1 - \tfrac{ a_0 a_3^2}{T^2} + a_0a_4 (\tfrac{S}{T}+\tfrac{a_2}{T}) \tfrac{1}{T}=\frac{(S+a_2) (2S-a_2)^2}{T^2}, \]
we see that $\tfrac{(S+a_2) (2S-a_2)^2}{T^2}_{|P_\infty}=1$. Therefore, $\tfrac{W}{\sqrt a_0 Y^2}_{|P_\infty}=1$. This shows that 
$\psi'(P_\infty)=P_+$.\\

Near $Q_+$, we express $\psi'$ in terms of the coordinates 
$(W/Y^2,1/Y)$ around $P_-$ by
\[ (\tfrac{W}{Y^2},\tfrac{1}{Y})= (\tfrac{4 \sqrt a_0 (2S-a_2)(S+a_2)^2}{(T + \sqrt a_0 a_3)^2}-\sqrt a_0,
\tfrac{2 \sqrt a_0 (S+a_2)}{T+ \sqrt a_0 a_3}).\]
Evaluating this at $(S,T)=(-a_2, \sqrt a_0 a_3)$ we get $(\tfrac{W}{Y^2},\tfrac{1}{Y})=(-\sqrt a_0,0)$ which shows that $\psi'(Q_+)=P_-$.\\

It is readily checked that $\psi' \circ \psi$ is the identity map on $\mE_L$ and $\psi \circ \psi'$ is the identity map on $E_L$, and hence $\mE_L$ and $E_L$ are isomorphic over $\bF_q$. We conclude that $\nu_L=\#\mE_L(\bF_q)/2 =\#E_L(\bF_q)/2$.
\eep
\begin{rem} \label{rem2}
    For nonzero $\lambda \in \bF_q$, multiplying \eqref{eq:E_L_def} by $\lambda^6$ gives the equation 
\[ (\lambda^3T) ^2=4 (\lambda^2 S)^3  - \lambda^4 g_2(L) (\lambda^2 S) - \lambda^6 g_3(L),\] which shows that the elliptic curve  obtained by replacing $(g_2(L), g_3(L)$ in the equation of $E_L$ by $(\lambda^4 g_2(L), \lambda^6 g_3(L))$ is isomorphic to $E_L$.
Using this in \eqref{eq:IJD_L}-\eqref{eq:g2g3def} with $\lambda = \det(g)^3$ for $g \in GL_2(q)$,  we see that the elliptic curve $E_L$ and $E_{g \cdot L}$ are isomorphic over $\bF_q$. In particular, $\# E_L(\bF_q)$ only depends on the $G$-orbit of $L$.
\end{rem}

\begin{rem} \label{rem1} For later use, we show that $\#E_L(\bF_q)$ is divisible by $3$. The equation of the elliptic curve $E_L$ can  be rewritten as
\[E_L:  T^2 - \tfrac{1}{4} (J_\varphi-\sqrt{I_\varphi}^3)^2
 = (S- \tfrac{3 I_\varphi}{4})(4 S^2+ 3 I_\varphi S - \tfrac{3}{2} I_\varphi^2 - 3 J_\varphi \sqrt I_\varphi),
\]
which shows that the points 
\[(S,T) = \left(\tfrac{3I_\varphi}{4}, \pm \tfrac{J_\varphi - \sqrt{I_\varphi}^3}{2}\right), \]
lie on $E_L$. It is easy to verify that these points are 
flex points of $E_L$, and hence $3$-torsion points of the group $E_L(\bF_q)$. Therefore, $\#E_L(\bF_q)$ is divisible by $3$.    
\end{rem}

 \section{Solution of Problem \ref{prob1}} \label{S6}
We first discuss the solution of Problem \ref{prob1} for the  ten $G$-orbits of non-generic lines. There are ten $G$-orbits of the non-generic lines of $PG(3,q)$,  the description of which can be found in \cite{BPS,DMP1,GL}. In the next lemma, for our purposes, we give a list of those orbits with their associated binary quartic forms. The symbol $\ep$ denotes a fixed quadratic non-residue of $\bF_q$.

\begin{lem} \label{nongeneric}\cite[Lemma 6.1]{KPP}
The set of non-generic lines decompose into the following ten orbits:
\begin{enumerate}
\item The orbit $\fO_2 = \pi^{-1}(G \cdot X^4)$  consists of the tangent lines to $C$, and is represented by $(z_0,\dots,z_5)=(0,0,0,0,1,0)$.
\item the orbit $\fO_4 = \pi^{-1}(G \cdot X^3Y)$ consists of the non-tangent unisecants contained in osculating planes of $C$, and is represented by $(z_0,\dots,z_5)=(0,0,0,1,0,0)$.

\item orbits $\fO_1$ and $\fO_1^{\perp}$ from $\pi^{-1}(G \cdot  X^2Y^2)$ of size $(q^2+q)/2$ each  and consisting of the secant lines, and the real axes of $C$, respectively. They
 are represented by $(z_0,\dots,z_5)=(0,0,1,0,0,1)$ and $(0,0,1,0,0,-1)$ respectively.
 \item orbits  $\fO_3$ and $\fO_3^{\perp}$ from $\pi^{-1}(G \cdot (X^2-\epsilon Y^2)^2$  of size $(q^2+q)/2$ each  and consisting of the imaginary secant lines, and the imaginary axes of $C$ respectively. They are represented by $(z_0,\dots,z_5)=(\ep^2,0,-\ep/3,0,1,2\ep/3)$ and $(\ep^2,0,-\ep/3,0,1,-2\ep/3)$ respectively.
\item The class $\fO_5$ of  unisecants not lying in osculating planes consists of the two orbits $\fO_{51}$ and $\fO_{52}$ below.  The class  of (non-axes) external lines in osculating planes consists of the two orbits $\fO_{51}^{\perp}$ and $\fO_{52}^{\perp}$ below.
\begin{enumerate}
\item orbits  $\fO_{51}$  and $\fO_{51}^{\perp}$ from $\pi^{-1}(G \cdot X^2(X^2-\epsilon Y^2))$ of size $(q^3-q)/2$ each. They
 are represented by $(z_0,\dots,z_5)=(0,0,\ep,0,-6,\ep)$ and $(0,0,\ep,0,-6,-\ep)$ respectively.
\item orbits  $\fO_{52}$  and $\fO_{52}^{\perp}$ from $\pi^{-1}(G \cdot X^2Y(Y-X))$  of size $(q^3-q)/2$ each. They
 are represented by $(z_0,\dots,z_5)=(0,0,2,3,0,2)$ and $(0,0,2,3,0,-2)$ respectively.
\end{enumerate}
% are represented by by $(z_0,\dots,z_5)=(\ep^2,0,-\ep/3,0,1,2\ep/3)$ and (\ep^2,0,-\ep/3,0,1,-2\ep/3)$ respectively.
%The orbit $\fO_{51}$ together with the orbit $\fO_{52}$ below comprise the set $\fO_5$ unisecants not contained in osculating planes.
%Similarly the set of external lines in osculating planes consists of two orbits $\fO_5^{\perp}= \fO_{51}^{\perp} \cup \fO_{52}^{\perp}$

\end{enumerate}
\end{lem}

The solution of Problem \ref{prob1} for these 10 orbits found by Davydov, Marcugini, and Pambianco \cite{DMP3} and G\"unay and Lavrauw in \cite{GL} is summarized in the Proposition below. We now give a quick proof of these results.

\begin{prop} \label{nongenericresult3}
Let $L$ be a non-generic line of $PG(3,q)$ and let $S$ denote the set of points of $L$. For $L\in \{\fO_1,\fO_1^\perp,\fO_2,\fO_3,\fO_3^\perp,\fO_4,\fO_{51},\fO_{51}^\perp, \fO_{52},\fO_{52}^\perp\}$, the numbers $|S\cap O_i|$, for $i=1,\dots,5$ can be given by the following table:
    \begin{table}[h] 
\begin{tabular}{c| *{5}{c}}
Orbit   & $|O_1\cap S|$ & $|O_2\cap S|$ & $|O_3\cap S|$ & $|O_4\cap S|$ & $|O_5\cap S|$    \\
&&&&&\\  \hline \\
 $\fO_1$   & $2$  &$0$ &  $\tfrac{(\mu+1)(q-1)}{6}$  & $\tfrac{(1-\mu)(q-1)}{2}$  &  $\tfrac{(\mu+1)(q-1)}{3}$     \\ 
 $\fO_1^{\perp}$   & $0$ & $2$  & $(q-1)$  &$0$  &$0$       \\ 
 $\fO_2$   & $1$ &$q$  &$0$   &$0$  &$0$       \\ 
 $\fO_3$   & $0$ & $0$ & $\tfrac{(1-\mu)(q+1)}{6}$  &$\tfrac{(1+\mu)(q+1)}{2}$  &$\tfrac{(1-\mu)(q+1)}{3}$       \\ 
 $\fO_3^{\perp}$   &$0$  &$0$  & $0$  &$(q+1)$  &$0$       \\ 
 $\fO_4$   &$1$  &$1$  & $\tfrac{q-1}{2}$   & $\tfrac{q-1}{2}$  & $0$       \\ 
 $\fO_{51}$   & $1$  & $0$ & $\tfrac{q-\mu}{6}$  &$\tfrac{q+\mu}{2}$  & $\tfrac{q-\mu}{3}$      \\ 
 $\fO_{51}^{\perp}$   & $0$  & $1$ &$\tfrac{q-1}{2}$   & $\tfrac{q+1}{2}$ &$0$       \\ 
 $\fO_{52}$   &$1$  &$2$  &$\tfrac{q-\mu-6}{6}$   & $\tfrac{q+\mu-2}{2}$ & $\tfrac{q-\mu}{3}$      \\ 
 $\fO_{52}^{\perp}$   & $0$ & $3$ &   $\tfrac{q-3}{2}$ &  $\tfrac{q-1}{2}$ &$0$       \\ 
  \\ 
\end{tabular}
\end{table}
\end{prop}
\bep
Of the $10$ orbits of non-generic lines, the orbits of lines contained in osculating planes of $C$ are: \\

(i)  Let  $L\in \fO_2$ (tangent lines of $C$) represented by the pencil $t \mapsto X^2(X+tY)$ which has $1$ point $t=\infty$ in $O_1$ and remaining $q$ points in $O_2$.\\

(ii) Let $L\in\fO_4$ (non-tangent unisecants in osculating planes of $C$) represented by the pencil $t \mapsto X(Y^2-tX^2)$ which has one point $t=\infty$ in $O_1$, one point $t=0$ in $O_2$, $(q-1)/2$ points each in $O_3$ and $O_4$.\\

(iii) Let $L\in\fO_1^\perp$ (real axes of $C$) represented by the pencil $t \mapsto XY(Y+tX)$ which has two  points $t=0,\infty$ in $O_2$ and $(q-1)$  points in $O_3$.\\

(iv) Let $L\in\fO_3^\perp$ (imaginary axes of $C$) represented by the pencil $t \mapsto (X^2-\ep Y^2)(Y+tX)$ which has all $(q+1)$  points  in $O_4$.

(v)-(vi) (external lines  in osculating planes)  \\

(v) Let $L\in \fO_{51}^\perp$ represented by the pencil $t \mapsto XY(Y+tX)$ which has two  points $t=0,\infty$ in $O_2$ and $(q-1)$  points in $O_3$.\\

(vi) Let $L\in \fO_{52}^\perp$ represented by the pencil $t \mapsto X((X+Y)^2-tY^2)$ which has three  points $t=0,1,\infty$ in $O_2$, $(q-3)/2$  points in $O_3$, and $(q-1)/2$ points in $O_4$.\\

For the remaining $4$ orbits $\fO_1, \fO_3,\fO_{51}, \fO_{52}$, we will use Proposition \ref{hLXY} and Lemma \ref{h_lem}. For $L$ representing such an orbit, we  need to determine the sizes of the sets $a_1(L), a_3(L), a_{3,1}(L), a_{3,3}(L)$ as defined in the Lemma \ref{h_lem}.\\

(vii) Let $L\in \fO_1$ (real secants of $C$) with associated binary quartic form $\varphi_L(X,Y)=X^2Y^2$. Here 
\begin{align*}
    \varphi_{s,t}(X,Y)& = (Xt-Ys) h^L_{(s,t)}(X,Y)\\
    &=(Xt-Ys) \bbsm X^2 &  XY & Y^2 \besm \bbsm 0 & 0 & 2 \\ 0 & 2 & 0 \\ 2 & 0 & 0 \besm \bbsm s^2 \\ st \\ t^2 \besm \\
    &=2(Xt-Ys) (Ys^2+XYst+Xt^2) \\
    &= 2(Xt-Ys) (Xt-Y\omega s)(Xt-Y \omega^2s),
\end{align*}
where $\omega$ is a primitive cube root of unity. Here $a_1(L)=\{(0,1),(1,0)\}$.
If $q \equiv 1 \mod 3$ then $\omega \in \bF_q$ and hence there are $a_{3,3}(L)=\bF_q^{\times}$ whereas $a_{3,1}(L)= a_2(L) = \emptyset$.   If $q \equiv 2 \mod 3$ then $\omega \notin \bF_q$ and hence  $a_{3,1}(L)=\bF_q^{\times}$ whereas $a_{3,3}(L)= a_{2}(L)= \emptyset$. Therefore, $|S\cap O_1|=2$, $|S\cap O_2|=0$, $|S\cap O_3|=\tfrac{(\mu+1)(q-1)}{6}$, $|S\cap O_4|=\tfrac{(1-\mu)(q-1)}{2}$, and $|S\cap O_5|=\tfrac{(\mu+1)(q-1)}{3}$.\\

(viii) Let $L\in \fO_3$ (imaginary secants of $C$) with associated binary quartic form $\varphi_L(X,Y)=(X^2-\ep Y^2)^2$. Here 
\begin{align*}
    \varphi_{s,t}(X,Y)&= (Xt-Ys) h^L_{(s,t)}(X,Y)\\
    &=(Xt-Ys) \bbsm X^2 &  XY & Y^2 \besm \bbsm 1 & 0 & \tfrac{\ep}{3} \\ 0 & -\tfrac{\ep}{3} & 0 \\ \tfrac{\ep}{3} & 0 & \ep^2 \besm \bbsm s^2 \\ st \\ t^2 \besm \\
    &=(Xt-Ys) \left( X^2(s^2+\tfrac{\ep t^2}{3})-\tfrac{\ep}{3}XYst+Y^2(\tfrac{\ep s^2}{3}+\ep^2t^2) \right) \\
    &=(Xt-Ys) \left( (Xs-Y\ep t)^2 - \tfrac{\ep}{-3} (Xt-Ys)^2 \right).
\end{align*}
We note that $\tfrac{\ep}{-3}$ is a square in $\bF_q$ if and only if $q \equiv 2\mod 3$.
Thus if $q \equiv 1 \mod 3$ then $a_{3,1}(L)=PG(1,q)$ and $a_1(L)= a_2(L)= a_{3,3}(L)= \emptyset$. If $q \equiv 2 \mod 3$ then $a_{3,3}(L)=PG(1,q)$ where as $a_1(L)= a_2(L)= a_{3,1}(L)= \emptyset$. Therefore, $|S\cap O_1|=0$, $|S\cap O_2|=0$, $|S\cap O_3|=\tfrac{(1-\mu)(q+1)}{6}$, $|S\cap O_4|=\tfrac{(1+\mu)(q+1)}{2}$, and $|S\cap O_5|=\tfrac{(1-\mu)(q+1)}{3}$.\\

(ix)-(x) (unisecants not in osculating planes). \\ 

(ix) Let $L\in \fO_{51}$ with associated binary quartic form $\varphi_L(X,Y)=X^2(X^2-\ep Y^2)$. Here 
\begin{align*}
    \varphi_{s,t}(X,Y)&= (Xt-Ys) h^L_{(s,t)}(X,Y)\\
    &=(Xt-Ys) \bbsm X^2 &  XY & Y^2 \besm \bbsm -6 & 0 & 2\ep \\ 0 & 2\ep & 0 \\ 2\ep & 0 & 0 \besm \bbsm s^2 \\ st \\ t^2 \besm \\
    &=2\ep (Xt-Ys) (s^2Y^2 + st XY+ X^2(t^2-\tfrac{3}{\ep}s^2)) .
\end{align*}
If $(s,t)=(0,1)$ then $\varphi_{s,t}=X^3$. If $(s,t)=(1,t)$ then 
\[ \tfrac{\varphi_{1,t}(X,Y)}{Xt-Y}=\tfrac{\ep}{2}
\left( 2Y + X\left[ t+ 2\sqrt{\tfrac{3}{\ep}- \tfrac{3t^2}{4} } \right] \right)
 \left(2Y + X \left[ t- 2\sqrt{\tfrac{3}{\ep}- \tfrac{3t^2}{4}}\right] \right) .\]
Since $\tfrac{t^2}{4}-\tfrac{1}{\ep}\neq 0$ for all $t\in \bF_q$, we have $a_2(L)=\emptyset$. Among the values of $t\in \bF_q$, the expression $(\tfrac{t^2}{4}-\tfrac{1}{\ep})$ is a square for $\tfrac{q-1}{2}$ choices of $t$, and a non-square for $\tfrac{q+1}{2}$ choices of $t$.
Therefore, the number of values of $t\in \bF_q$  for which $\sqrt{-3(\tfrac{t^2}{4}-\tfrac{1}{\ep})}\in \bF_q$ is $\tfrac{q-\mu}{2}$. Thus $a_1(L)=\{(0,1)\}$, $|a_{3,3}(L)|= \tfrac{q-\mu}{2}$ and $|a_{3,1}(L)|=q-\tfrac{q-\mu}{2}=\tfrac{q+\mu}{2}$. Therefore, $|S\cap O_1|=1$, $|S\cap O_2|=0$, $|S\cap O_3|=\tfrac{q-\mu}{6}$, $|S\cap O_4|=\tfrac{q+\mu}{2}$, and $|S\cap O_5|=\tfrac{q-\mu}{3}$.\\

(x) Let $L\in \fO_{52}$ with associated binary quartic form $\varphi_L(X,Y)=X^2Y(Y-X)$. Here 
\begin{align*}
    \varphi_{s,t}(X,Y)&= (Xt-Ys) h^L_{(s,t)}(X,Y)\\
    &=(Xt-Ys) \bbsm X^2 &  XY & Y^2 \besm \bbsm 0 & -6 & 4 \\ -6 & 4 & 0 \\ 4 & 0 & 0 \besm \bbsm s^2 \\ st \\ t^2 \besm \\
    &=2(Xt-Ys)(2s^2Y^2 + s XY(2t-3s)+ X^2t(2t-3s) . 
\end{align*}
If $(s,t)=(0,1)$, then $\varphi_{s,t}=X^3$. If $(s,t)=(1,t)$, $t\in \bF_q$, then $\tfrac{\varphi_{1,t}(X,Y)}{Xt-Y}$ equals
{\small \[\tfrac{1}{4}\left(4Y - X\left[3-2t+ \sqrt{-3\{ (2t-1)^2-4 \} }\right] \right) \left(4Y - X\left[3-2t - \sqrt{-3\{ (2t-1)^2-4 \} }\right]\right)  .\]}
Here $a_2(L)=\{0,1,3/2,-1/2\}$, and  $a_1(L)=\{(0,1)\}$. We have $a_{3,3}(L)=\{t\in \bF_q \colon -3((2t-1)^2-4) \text{ is a non-zero square}\}$ and $a_{3,1}(L)=\{t\in \bF_q \colon -3((2t-1)^2-4) \text{ is a non-square}\}$. Thus $|a_{3,1}(L)|+|a_{3,3}(L)|=q-4$.

The number of solutions $(t,\lambda)$ in $\bF_q$ of the equation $-3 ((2t-1)^2-4)=\lambda^2$ is $q-1$, if $q\equiv 1 \mod 3$ and $q+1$, if $q\equiv -1 \mod 3$. Among these we also have $6$ solutions  $(0,\pm 1),(1,\pm 1),(-1/2,0),(3/2,0)$ with $t=0,1,3/2,-1/2$. This gives $\tfrac{q-1-6}{2}=\tfrac{q-7}{2}$ and $\tfrac{q+1-6}{2}=\tfrac{q-5}{2}$ values of $t\in a_{3,3}(L)$ in the cases $q\equiv 1$ and  $q\equiv -1 \mod 3$, respectively. Thus we have $|a_{3,3}(L)|=\tfrac{q-\mu-6}{2}$ and $|a_{3,1}(L)|=q-4-\tfrac{q-\mu-6}{2}=\tfrac{q+\mu-2}{2}$. Hence, $|S\cap O_1|=1$, $|S\cap O_2|=2$, $|S\cap O_3|=\tfrac{q-\mu-6}{6}$, $|S\cap O_4|=\tfrac{q+\mu-2}{2}$, and $|S\cap O_5|=\tfrac{q-\mu}{3}$.
\eep

\subsection{Proof of main theorem:}
Let $L$ be a generic line and $\varphi_L$ be the quartic form associated with it. Let $S$ denote the set of $(q+1)$ points of $L$. By Theorem \ref{thm_elliptic}, $\nu_L =  (\# E_L(\bF_q)-\eta(L))/2$, where $E_L$ is the elliptic curve associated with $L$ as defined in (\ref{eq:E_L_def}). Let 
\[\eta_L:=\begin{cases} i &\text{if $\varphi_L \in \mF_i$ for $i=1,2,4$}\\
 0 &\text{if $\varphi_L$ is in $\mF_2'$ or $\mF_4'$}, \end{cases}\]
as defined in (\ref{eq:eta_def}).
Then, by Proposition \ref{gen_prop}, we have 
 \begin{enumerate}
        \item $|S\cap O_1|=0$,
        \item $|S\cap O_2|=\eta_L$,
        \item $|S\cap O_3|=\tfrac{1}{3} \left( \tfrac{\# E_L(\bF_q)-\eta(L)}{2}-\eta_L \right)=\tfrac{\# E_L(\bF_q)-3\eta_L}{6}$,
        \item $|S\cap O_4|=q+1- \tfrac{\# E_L(\bF_q)-\eta(L)}{2}-\eta_L=q+1-\frac{\#E_L(\bF_q)+\eta_L}{2}$, \item $|S\cap O_5|= \tfrac{1}{3}\left(2\tfrac{\# E_L(\bF_q)-\eta(L)}{2}+\eta_L\right)=\tfrac{\# E_L(\bF_q)}{3}$. 
    \end{enumerate}
 % If $\varphi_L\in \mF_2' \cup \mF_4'$, then $\eta_L=0$, and therefore,  $\nu_L = \# E_L(\bF_q)/2$. Then, by Proposition \ref{gen_prop}, we have 
 % \begin{enumerate}
 %        \item $|S\cap O_1|=|S\cap O_2|=0$,
 %        \item $|S\cap O_3|=\tfrac{\# E_L(\bF_q)}{6}$,
 %         \item $|S\cap O_4|=q+1-\frac{\# E_L(\bF_q)}{2}$, \item $|S\cap O_5|=\tfrac{\# E_L(\bF_q)}{3}$. 
 %         \end{enumerate}
This proves the theorem.
\hfill $\square$

%\bibliography{incidenceP3}
%\bibliographystyle{amsplain}

\bibliographystyle{amsplain}

\end{document}